\documentclass[12pt]{amsart}
\usepackage{a4}
\usepackage{amsmath,amssymb,amscd,amsthm}
\usepackage{amsfonts}
\usepackage[dvipdfm]{color}

\newtheorem{thm}{Theorem}[section]
\newtheorem{prop}[thm]{Proposition}
\newtheorem{lem}[thm]{Lemma}

\newtheorem*{main}{Main Theorem}
\newtheorem*{Q}{Question}

\theoremstyle{definition}
\newtheorem{defi}[thm]{Definition}
\newtheorem{rem}[thm]{Remark}

\newcommand{\g}{\mathfrak{g}}
\newcommand{\h}{\mathfrak{h}_3}
\newcommand{\m}{\mathfrak{m}}
\newcommand{\M}{\widetilde{\mathfrak{M}}}
\newcommand{\PM}{\mathfrak{PM}}

\newcommand{\A}{\mathbb{R}^{\times}\mathrm{Aut}(\g)}
\newcommand{\Aut}{\mathrm{Aut}(\g)}
\newcommand{\D}{\mathbb{R} \oplus {\rm Der}(\mathfrak{g})}
\newcommand{\Der}{\mathrm{Der}(\g)}
\newcommand{\naiseki}{\langle , \rangle}
\newcommand{\R}{\mathbb{R}}
\newcommand{\GL}{\mathrm{GL}_n(\R)} 
 
\numberwithin{equation}{section}

\author{Takahiro Hashinaga}
\address{Department of Mathematics, Hiroshima University, 
Higashi-Hiroshima 739-8526, Japan}
\email{hashinaga@hiroshima-u.ac.jp}

\author{Hiroshi Tamaru}
\address{Department of Mathematics, Hiroshima University, 
Higashi-Hiroshima 739-8526, Japan}
\email{tamaru@math.sci.hiroshima-u.ac.jp}

\keywords{Lie groups, 
left-invariant Riemannian metrics, 
solvsolitons, symmetric spaces, 
minimal submanifolds.} 
\thanks{2010 \textit{Mathematics Subject Classification}. 
53C30, 53C25.}

\thanks{The first author was supported in part by Grant-in-Aid for JSPS Fellows (11J05284). 
The second author was supported in part by KAKENHI (20740040, 24654012).} 

\date{}

\title[Solvsolitons and the corresponding submanifolds]
{Three-dimensional solvsolitons and the minimality of the corresponding submanifolds}

\begin{document}

\maketitle

\begin{abstract}
In this paper, we define 
the corresponding submanifolds to left-invariant Riemannian metrics on Lie groups,  
and study the following question: 
does a distinguished left-invariant Riemannian metric on a Lie group 
correspond to a distinguished submanifold? 
As a result, we prove that the solvsolitons on 
three-dimensional simply-connected solvable Lie groups 
are completely characterized by the minimality of the corresponding submanifolds. 

\end{abstract}

\section{Introduction}

\subsection{Solvsolitons}

Lie groups with left-invariant Riemannian metrics provide 
a lot of concrete examples of distinguished Riemannian metrics, 
such as Einstein metrics and Ricci solitons. 
Recently, such distinguished left-invariant Riemannian metrics 
have been studied very actively 
(see, for instance, 
\cite{F, Heber, J, J11, Lau01, Lau09, Lau10, Lau11, L-W, N06, N07, T05, T08, T11, W03, W11}). 

In this paper,
we treat solvsolitons as distinguished left-invariant Riemannian metrics.
Recall that a left-invariant Riemannian metric $\naiseki$ 
on a simply-connected solvable Lie group $G$ is called a \textit{solvsoliton} 
if the Ricci operator satisfies 
\begin{align}
\label{eq:solvsoliton}
\mathrm{Ric}_{\naiseki} =c I + D \quad 
(\mbox{for some $c \in \mathbb{R}$ and $D \in \Der$}) .
\end{align}
A solvsoliton on $G$ is called a \textit{nilsoliton} 
if $G$ is nilpotent. 
Solvsolitons have been introduced by Lauret (\cite{Lau11}), 
and play a key role in the study of homogeneous Ricci solitons. 
In particular, 
every solvsoliton on a simply-connected solvable Lie group is a Ricci soliton 
(\cite{Lau11}), 
and every left-invariant Ricci soliton on a solvable Lie group 
is isometric to a solvsoliton (\cite{J11}). 

In the study of solvsolitons, 
including left-invariant Einstein metrics on solvable Lie groups, 
the tools from geometric invariant theory have played very important roles. 
Among others, 
Lauret (\cite{Lau11})
obtained structural and uniqueness results for solvsolitons. 
It enables to classify 
solvsolitons in low-dimensional cases (\cite{Lau11, W11}). 
For further information, we refer to \cite{Lau09} and references therein. 

\subsection{An approach from the submanifold theory}

In this paper, we propose a new framework 
for studying distinguished left-invariant Riemannian metrics, 
such as solvsolitons, 
in terms of the group actions on and the submanifold theory 
in noncompact symmetric spaces. 
This paper only concerns simply-connected solvable Lie groups 
of dimension three, 
but here we formulate our framework in a general way. 

Let $G$ be a Lie group and $\g$ be the Lie algebra of $G$. 
Consider the set of all left-invariant Riemannian metrics on $G$, 
which can be identified with 
\begin{align}
\M := \{ \naiseki \mid 
\mbox{an inner product on $\g$} 
\} 
\cong \mathrm{GL}_n(\mathbb{R})/\mathrm{O}(n) , 
\end{align} 
where $n = \dim G$. 
Throughout this paper, this space is assumed to be endowed with 
the natural $\mathrm{GL}_n(\mathbb{R})$-invariant Riemannian metric 
(see Subsection~2.1), 
and hence is a noncompact symmetric space. 
Let us consider the actions of 
\begin{align}
\A := \{ 
c \varphi \in \mathrm{GL}_n(\mathbb{R}) \mid 
c \in \mathbb{R}^\times , \ \varphi \in \mathrm{Aut}(\g) 
\} 
\end{align} 
on 
$\M = \mathrm{GL}_n(\mathbb{R})/\mathrm{O}(n)$. 
Note that $\mathbb{R}^\times$ denotes the set of nonzero scalar maps on $\g$, 
and $\mathrm{Aut}(\g)$ the automorphism group. 
The group $\A$ comes from 
the equivalence relation ``isometry up to scaling" in the Lie algebra level 
(see Definition \ref{I}). 
Denote its equivalence class by $[~]$. 
Then, for each inner product $\naiseki$, 
it follows from \cite{KTT} that 
\begin{align}
[\naiseki] = \A . \naiseki , 
\end{align} 
which we call 
\textit{the corresponding submanifold} to $\naiseki$. 
An important point is that the Riemannian geometric properties of $\naiseki$ 
are preserved by isometry and scaling. 
Thus we can regard properties of left-invariant Riemannian metrics 
as properties of the corresponding submanifolds. 
Therefore, it would be natural to ask the following: 
\begin{Q}
Does a distinguished left-invariant Riemannian metric correspond to 
a distinguished submanifold? 
\end{Q}

If an answer for this question is positive, 
then the approach from the corresponding submanifolds would possibly be useful for the study of left-invariant metrics. 
For example, 
the existence and nonexistence problem 
of distinguished left-invariant Riemannian metrics on $G$ 
can be translated to the problem of the $\A$-action, that is, 
the existence and nonexistence of distinguished orbits. 

\subsection{Results of this paper}

Let $G$ be a three-dimensional simply-connected solvable Lie group 
with Lie algebra $\g$. 
In this paper, 
we present that there is a good relationship between the existence 
of solvsolitons on $G$ and geometric 
aspects of the corresponding action of $\A$ on $\mathrm{GL}_3(\mathbb{R})/\mathrm{O}(3)$. 
We will see this by using 
the classification of three-dimensional solvable Lie algebras (\cite{ABDO}), 
which is summarized in Table~\ref{table:three-dimensional}. 
Note that Table~\ref{table:three-dimensional} 
contains a decomposable one, $\mathfrak{r}_{3, 0}$, 
and our results are true for both decomposable and indecomposable cases. 

\begin{table}[ht]
\begin{tabular}{|c|l|c|} \hline
Name & Non-zero commutation relation & \\ \hline
$\mathfrak{h}_3$ & $[e_1,e_2]=e_3$ & Nilpotent\\
$\mathfrak{r}_3$ & $[e_1,e_2]=e_2+e_3,~[e_1,e_3]=e_3$ & Solvable\\
$\mathfrak{r}_{3, a}$ & $[e_1,e_2]=e_2,~[e_1,e_3]=ae_3$ \hspace{3mm} $(-1 \leq a \leq 1)$ & Solvable \\
$\mathfrak{r}^{\prime}_{3, a}$ & $[e_1,e_2]=ae_2 - e_3$, $[e_1,e_3]=e_2 + ae_3$ \hspace{3mm} $(a \geq 0)$ & Solvable \\ \hline
\end{tabular}  
\label{table:three-dimensional}
\caption{Three-dimensional solvable Lie algebras} 
\end{table} 

We recall that, 
for an isometric action on a Riemannian manifold, 
orbits of maximal dimension are said to be 
\textit{regular}, 
and other orbits 
\textit{singular}. 
An action is said to be of 
\textit{cohomogeneity one} if the regular orbits have codimension one. 
Then, the good relationship we obtain 
can be summarized as follows. 

\begin{itemize}
\item
Let $\g = \mathfrak{h}_3$ or $\mathfrak{r}_{3,1}$. 
Then, $\A$ acts transitively on $\M$, 
and hence 
there is the only one orbit.  
The left-invariant Riemannian metric on $G$ is 
unique up to isometry and scaling, 
and the metric is a solvsoliton 
(nilsoliton for $\mathfrak{h}_3$, and Einstein for $\mathfrak{r}_{3,1}$). 

\item 
Let $\g = \mathfrak{r}_3$. 
Then, the action of $\A$ is of cohomogeneity one, 
and all orbits are regular. 
Furthermore, all orbits are isometrically congruent to each other 
(namely there are no distinguished orbits). 
On the other hand, $G$ does not admit a solvsoliton. 

\item 
Let $\g=\mathfrak{r}_{3, a}$ $(-1 \leq a < 1)$. 
Then, the action of $\A$ is of cohomogeneity one, 
and all orbits are regular. 
This action has the unique minimal orbit. 
On the other hand, 
$G$ admits a solvsoliton, 
whose corresponding submanifold coincides with this minimal orbit. 

\item 
Let $\g = \mathfrak{r}^{\prime}_{3,a}$ ($a \geq 0$). 
Then, the action of $\A$ is of cohomogeneity one, 
and has the unique singular orbit. 
On the other hand, $G$ admits a left-invariant Einstein metric, 
whose corresponding submanifold coincides with this singular orbit. 
\end{itemize}

By studying the geometry of $\A$-orbits in more detail, 
we obtain a positive answer to the above mentioned 
Question for three-dimensional solvsolitons.
Namely, three-dimensional solvsolitons 
can completely be characterized by 
the minimality of the corresponding submanifold. 

\begin{main} 
\label{thm:main} 
Let $G$ be a three-dimensional simply-connected solvable Lie group, 
and $\naiseki$ be a left-invariant Riemannian metric on $G$. 
Then, $\naiseki$ is a solvsoliton if and only if 
the corresponding submanifold $[\naiseki]$ is a minimal submanifold 
in $\M$ with respect to the natural $\mathrm{GL}_3 (\mathbb{R})$-invariant Riemannian metric. 
\end{main}

This paper is organized as follows. 
In Section~2, 
we recall the necessary background on 
the corresponding submanifolds $[\naiseki]$ 
to left-invariant Riemannian metrics $\naiseki$ on Lie groups. 
In Section~3, 
for each three-dimensional solvable Lie algebra $\g$, we study the orbit space of the action of $\A$. 
Expressions of the orbit spaces will be used in both Sections~4 and 5. 
In Section~4, 
we study three-dimensional solvsolitons. 
In particular, we obtain the ``Milnor-type theorems'' for each $\g$, 
and apply them to the reclassification of three-dimensional solvsolitons. 
In Section~5, 
we study the actions of $\A$. 
The results of Sections~4 and 5 provide the proof of our Main Theorem. 

\section{The corresponding submanifolds} 

In this section, we define the notion of the corresponding submanifolds 
to left-invariant Riemannian metrics on Lie groups. 
This gives a correspondence between left-invariant Riemannian metrics and 
$\A$-homogeneous submanifolds. 

\subsection{The space of left-invariant metrics}

First of all, we recall the space of left-invariant Riemannian metrics, 
which will be the ambient space of the corresponding submanifolds. 
We refer to \cite{KTT}. 

Let $G$ be an $n$-dimensional simply-connected Lie group, 
and $\g$ be the Lie algebra of $G$. 
We consider the set of all left-invariant Riemannian metrics on $G$, 
which can naturally be identified with 
\begin{align}
\M := \{ \naiseki \mid \mbox{an inner product on $\g$} \} .  
\end{align}
We identify $\g$ with $\mathbb{R}^n$ as vector spaces from now on. 
Then, since $\mathrm{GL}_n(\mathbb{R})$ acts transitively on $\M$ by
\begin{align}
g.\langle \cdot , \cdot \rangle := \langle g^{-1} (\cdot) , g^{-1} (\cdot) \rangle 
\quad 
(\mbox{for} \  g \in \mathrm{GL}_n(\mathbb{R}), \ \naiseki \in \M ), 
\end{align}
we have an identification
\begin{align}
\M = \mathrm{GL}_n(\mathbb{R})/\mathrm{O}(n).
\end{align} 

Note that $\M$ equipped with the natural 
$\mathrm{GL}_n(\mathbb{R})$-invariant Riemannian metric 
is a noncompact Riemannian symmetric space. 
In order to describe this natural metric, 
we recall a general theory of reductive homogeneous spaces. 
Let $U/K$ be a reductive homogeneous space, 
that is, 
there exists an $\mathrm{Ad}_K$-invariant subspace $\m$ of $\mathfrak{u}$ satisfying 
\begin{align}
\label{eq:reductive}
\mathfrak{u} = \mathfrak{k} \oplus \mathfrak{m} . 
\end{align}
Note that $\mathfrak{u}$ and $\mathfrak{k}$ are the Lie algebras of $U$ and $K$, 
respectively, 
and $\oplus$ is the direct sum as vector spaces. 
The decomposition (\ref{eq:reductive}) is called a 
\textit{reductive decomposition}. 
Denote by $\pi : U \rightarrow U/K$ the natural projection, 
and by $o := \pi(e)$ the origin of $U/K$. 
We identify $\m$ with the tangent space $\mathrm{T}_o (U/K)$ at $o$ by 
\begin{align}
\mathrm{d} \pi_e |_{\m} : \m \rightarrow \mathrm{T}_o (U/K) . 
\end{align}
This identification induces 
a one-to-one correspondence between the set of 
$U$-invariant Riemannian metrics on $U/K$ 
and the set of $\mathrm{Ad}_K$-invariant inner products on $\m$. 

Now one can see that 
$\M = \mathrm{GL}_n(\mathbb{R})/\mathrm{O}(n)$ 
is a reductive homogeneous space, 
whose reductive decomposition is given by the subspace 
\begin{align}
\mathrm{sym}(n) := \{ X \in \mathfrak{gl}_n(\mathbb{R}) \mid X = {}^t X  \} . 
\end{align}
We define the $\mathrm{Ad}_{\mathrm{O}(n)}$-inner product on $\mathrm{sym}(n)$ by 
\begin{align}\label{met}
\langle X , Y \rangle := \mathrm{tr}(XY) 
\quad 
(\mbox{for} \  X,Y \in \mathrm{sym}(n)).
\end{align}
We call the 
$\mathrm{GL}_n(\mathbb{R})$-invariant Riemannian metric corresponding
to the above $\mathrm{Ad}_{\mathrm{O}(n)}$-inner product the \textit{natural Riemannian metric}. 

\subsection{The corresponding submanifolds} 

We now define the submanifolds 
in the space of left-invariant Riemannian metrics, 
and see that they are homogeneous. 
These submanifolds come from the equivalence relation 
``isometric up to scaling".  

\begin{defi}\label{I}
Two inner products $\naiseki_1$ and $\naiseki_2$ on $\g$ 
are said to be \textit{isometric up to scaling} 
if there exist $k > 0$ and an automorphism 
$f: \g \rightarrow \g$ 
such that $\langle \cdot , \cdot \rangle_1 = k \langle f (\cdot) , f (\cdot) \rangle_2$.  
\end{defi}

Assume that inner products $\naiseki_1$ and $\naiseki_2$ 
on $\g$ are isometric up to scaling. 
Then, the corresponding left-invariant Riemannian metrics on $G$, 
the simply-connected Lie group with Lie algebra $\g$, 
are isometric up to scaling as Riemannian metrics 
(we refer to \cite[Remark 2.3]{KTT}). 
Therefore, this equivalence relation preserves 
all Riemannian geometric properties of left-invariant metrics. 
In particular, it preserves solvsolitons. 

\begin{defi}
For each inner product $\naiseki$ on $\g$, 
we call its equivalence class 
$[\naiseki]$ the \textit{corresponding submanifold to} $\naiseki$. 
\end{defi}

Note that $[\naiseki]$ is a submanifold in 
$\M = \mathrm{GL}_n(\mathbb{R})/\mathrm{O}(n)$. 
We here recall that $[\naiseki]$ is a homogeneous submanifold.
Let us denote by 
\begin{align}
\mathbb{R}^{\times} & := \{ c \cdot \mathrm{id} : \g \rightarrow \g \mid 
c \in \mathbb{R} \setminus \{ 0 \} \} , \\ 
\Aut & := \{ \varphi : \g \rightarrow \g \mid \mbox{an automorphism} \} . 
\end{align}
Then, the subgroup $\A$ 
of $\mathrm{GL}_n(\mathbb{R})$ acts naturally on $\M$. 
Let us denote by $\A.\naiseki$ the $\A$-orbit through $\naiseki$. 

\begin{prop}
[{\cite[Theorem 2.5]{KTT}}] 
Let $\naiseki$ be an inner product on $\g$. 
Then, the corresponding submanifold 
$[ \naiseki ]$ is a homogeneous submanifold with respect to $\A$, that is, 
\begin{align}
[\naiseki] = \A.\naiseki. 
\end{align} 
\end{prop}\label{orbit}

\section{Explicit expressions of the moduli spaces}

In this section, for each three-dimensional solvable Lie algebra $\g$, 
we give an explicit expression of 
the ``moduli space'' of left-invariant Riemannian metrics. 
The results of this section will be used in Sections~4 and 5. 

\subsection{Preliminaries on the moduli spaces}

In this subsection, 
we recall some necessary facts on the moduli spaces 
of left-invariant Riemannian metrics. 
We refer to \cite{KTT}. 

\begin{defi}\label{DefPM}
For a Lie algebra $\g$, 
the quotient space of $\M$ by ``isometric up to scaling" is called the 
\textit{moduli space of left-invariant Riemannian metrics}, 
and denoted by 
\begin{align}
\PM := \{ [\naiseki] \mid \naiseki \in \M \} . 
\end{align}
\end{defi}

In order to determine $\PM$ explicitly, 
we will use the following notion of a set of representatives. 
Recall that we identify $\g \cong \R^n$. 
Denote by $\{ e_1, \ldots, e_n \}$ the canonical basis of $\R^n$, 
and by $\naiseki_0$ the inner product so that the canonical basis is orthonormal.

\begin{defi}
A subset $U \subset \GL$ is called a 
\textit{set of representatives} of $\PM$ if it satisfies 
\begin{align}
\PM = \{ [h. \naiseki_0] \mid h \in U \} . 
\end{align}
\end{defi}

In the later arguments, it is convenient to use the double cosets. 
Note that our double coset $[[g]]$ of $g \in \mathrm{GL}_n(\mathbb{R})$ 
is defined by 
\begin{align} 
[[ g ]] := \mathbb{R}^{\times} \Aut \cdot g \cdot \mathrm{O}(n) . 
\end{align} 

\begin{lem}[\cite{HTT}]
\label{lem:representatives}
Let $U \subset \GL$. 
Then, 
$U$ is a set of representatives of $\PM$ if and only if, 
for every $g \in \GL$, 
there exists $h \in U$ such that $h \in [[g]]$. 
\end{lem} 

In order to obtain a set of representatives of $\PM$, one needs $\A$. 
The Lie algebra of $\A$ coincides with $\D$, where 
\begin{align}
\mathbb{R} 
& := \{ c \cdot \mathrm{id} : \g \rightarrow \g \mid c \in \mathbb{R} \} , \\ 
\Der 
& := \{ D \in \mathfrak{gl}(\g) \mid 
D[\cdot , \cdot ]=[D (\cdot ),\cdot ]+[\cdot , D (\cdot )] \} . 
\end{align}
The Lie algebra $\D$ determines $(\A)^0$, 
the connected component of $\A$ containing the identity. 

For each three-dimensional solvable Lie algebra, 
the moduli space $\PM$ has been studied in \cite{KTT}. 
We here mention the trivial case, 
which means that $\PM$ consists of one point. 

\begin{prop}[\cite{KTT, Lau03}]\label{tri}
Let $\g = \mathfrak{h}_3$ or $\mathfrak{r}_{3, 1}$. 
Then, $\A$ acts transitively on $\M$, and hence $\PM = \{ \mathrm{pt} \}$. 
\end{prop}

\begin{rem}\label{rem-trivial case}
One can see that Theorem~\ref{thm:main} holds 
for $\g = \mathfrak{h}_3$ and $\mathfrak{r}_{3, 1}$. 
In fact, 
it is well-known that any left-invariant Riemannian metrics $\naiseki$ on 
these Lie algebras are solvsolitons 
(nilsoliton for $\mathfrak{h}_3$, and Einstein for $\mathfrak{r}_{3,1}$). 
Furthermore, for every $\naiseki$, 
the corresponding submanifold $[\naiseki]$ coincides with the ambient space $\M$, 
which is minimal. 
\end{rem}

In the following, 
we will study the remaining three-dimensional solvable Lie algebras. 

\subsection{A lemma for nontrivial cases} 

This subsection gives a preliminary to obtain a set of representatives $U$ of $\PM$ for 
$\mathfrak{g} = \mathfrak{r}_3$, 
$\mathfrak{r}_{3, a}$ ($-1 \leq a <1$), and 
$\mathfrak{r}^{\prime}_{3, a}$ ($a \geq 0$). 

First of all, 
let us recall a matrix expression of $\Der$ for these Lie algebras. 
The following results can be calculated directly, 
and be found in \cite[Section~4]{KTT}. 

\begin{lem}[\cite{KTT}]
\label{der}
The matrix expressions of $\Der$ with respect to the bases $\{ e_1, e_2, e_3 \}$ 
in Table~\ref{table:three-dimensional} are given as follows$:$ 
\begin{enumerate}
\item Let $\g=\mathfrak{r}_{3}$. Then, we have
\begin{align*}
\displaystyle \Der = \left\{ \left(
\begin{array}{ccc}
0 & 0 & 0 \\
x_{21} & x_{22} & 0 \\
x_{31} & x_{32} & x_{22} 
\end{array}
\right) 
\mid x_{21}, x_{22}, x_{31}, x_{32} \in \R 
\right\} . 
\end{align*}
\item 
Let 
$\g=\mathfrak{r}_{3, a}$ $(-1 \leq a < 1)$.
Then, we have
\begin{align*}
\displaystyle \Der = \left\{ \left(
\begin{array}{ccc}
0 & 0 & 0 \\
x_{21} & x_{22} & 0 \\
x_{31} & 0 & x_{33} 
\end{array}
\right) \mid x_{21}, x_{22}, x_{31}, x_{33} \in \R \right\} . 
\end{align*}
\item 
Let 
$\g = \mathfrak{r}^{\prime}_{3, a}$ $(a \geq 0)$. 
Then, we have
\begin{align*}
\displaystyle \Der = \left\{ \left(
\begin{array}{ccc}
0 & 0 & 0 \\
x_{21} & x_{22} & -x_{23} \\
x_{31} & x_{23} & x_{22} 
\end{array}
\right) 
\mid x_{21}, x_{22}, x_{23}, x_{31} \in \R 
\right\} . 
\end{align*}
\end{enumerate}
\end{lem}

Let us consider $\A$ for these Lie algebras $\g$. 
One can see from Lemma~\ref{der} that $\A$ contain 
\begin{align}\label{f}
F := 
\left\{ \left(
\begin{array}{ccc}
x_{11} & 0 & 0 \\
x_{21} & x_{22} & 0 \\
x_{31} & 0 & x_{22}  
\end{array}
\right) \mid x_{11}, x_{22} > 0 \right\} . 
\end{align}
For the later use, we prepare the following lemma, 
which can be applied for all Lie algebras we have to consider. 

\begin{lem}\label{F}
Let $\g$ be a three-dimensional Lie algebra, and fix a basis of $\g$. 
If $F \subset \A$ holds, 
then the following $L^{\prime}$ is a set of representatives of $\PM$$:$ 
\begin{align}
L^{\prime} 
:= \left\{ \left(
\begin{array}{ccc}
1 & 0 & 0 \\
0 & 1 & 0 \\
0 & a_{32} & a_{33}
\end{array}
\right) \mid a_{33} > 0 \right\} . 
\end{align}
\end{lem}

\begin{proof}
Take any $g \in \mathrm{GL}_3(\mathbb{R})$. 
By Lemma~\ref{lem:representatives}, 
we have only to show that there exists $g^{\prime} \in L^{\prime}$ such that 
$g^{\prime} \in [[ g ]]$. 
First of all, one knows that there exists $k \in \mathrm{O}(3)$ such that 
\begin{align}
gk = \left(
\begin{array}{ccc}
g_{11} & 0 & 0 \\
g_{21} & g_{22} & 0 \\
g_{31} & g_{32} & g_{33}
\end{array}
\right) , \quad 
g_{11}, g_{22}, g_{33} > 0 . 
\end{align} 
By assumption, we can take 
\begin{align}
\varphi := 
\frac{1}{g_{11} g_{22}} 
\left( 
\begin{array}{rcc}
g_{22} & 0 & 0 \\
-g_{21} & g_{11} & 0 \\
-g_{31} & 0 & g_{11}  
\end{array}
\right) \in F \subset \A . 
\end{align}
By a direct calculation, one has
\begin{align}
[[ g ]] \ni \varphi g k = 
\left(
\begin{array}{ccc}
1 & 0 & 0 \\
0 & 1 & 0 \\
0 & g_{32}/g_{22} & g_{33}/g_{22} 
\end{array}
\right) =: g^{\prime} . 
\end{align}
Since $g^{\prime} \in L^{\prime}$, we complete the proof. 
\end{proof}

\subsection{Case of $\g=\mathfrak{r}_3$}

In this subsection, 
we give an explicit expression of $\PM$ for $\g = \mathfrak{r}_3$. 
We fix a basis $\{ e_1 , e_2 , e_3 \}$ of $\mathfrak{r}_3$ 
whose bracket relations are given by 
\begin{align}
[e_1 , e_2] = e_2 + e_3 , \quad [e_1 , e_3] = e_3 . 
\end{align} 

From Lemma~\ref{der}, we have 
\begin{align}\label{d1}
\displaystyle \D = \left\{ \left(
\begin{array}{ccc}
x_{11} & 0 & 0 \\
x_{21} & x_{22} & 0 \\
x_{31} & x_{32} & x_{22} 
\end{array}
\right) 
\mid x_{11}, x_{21}, x_{22}, x_{31}, x_{32} \in \R 
\right\} . 
\end{align}
This yields that 
\begin{align}\label{a1}
\displaystyle (\A)^0 = \left\{ \left(
\begin{array}{ccc}
x_{11} & 0 & 0 \\
x_{21} & x_{22} & 0 \\
x_{31} & x_{32} & x_{22} 
\end{array}
\right) \mid x_{11}, x_{22} > 0 \right\} . 
\end{align}
Therefore, we can apply Lemma~\ref{F} for this case. 

\begin{prop}\label{PM1}
Let $\g = \mathfrak{r}_3$. 
Then the following $U$ is a set of representatives of $\PM$$:$  
\begin{align}
\displaystyle U = \left\{  \left(
\begin{array}{ccc}
1 & 0 & 0 \\
0 & 1 & 0 \\
0 & 0 & 1/\lambda 
\end{array}
\right) \mid \lambda > 0 \right\} . 
\end{align}
\end{prop}

\begin{proof}
Take any $g \in \mathrm{GL}_3(\R)$. 
By Lemma~\ref{lem:representatives}, 
we have only to show that 
there exists $\lambda > 0$ such that 
\begin{align}
\left(
\begin{array}{ccc}
1 & 0 & 0 \\
0 & 1 & 0 \\
0 & 0 & 1/\lambda 
\end{array}
\right) \in [[g]] . 
\end{align}
We use $L^{\prime}$ defined in Lemma~\ref{F}. 
One has from (\ref{a1}) 
and Lemma~\ref{F} that there exists $g^{\prime} \in L^{\prime}$ such that $g^\prime \in [[g]]$. 
Since $g' \in L^{\prime}$, one can write 
\begin{align}
g^{\prime} = \left(
\begin{array}{ccc}
1 & 0 & 0 \\
0 & 1 & 0 \\
0 & a_{32} & a_{33} 
\end{array}
\right) , \quad a_{33} > 0 . 
\end{align} 
It follows from (\ref{a1}) that 
\begin{align}
\varphi := 
\left(
\begin{array}{ccc}
1 & 0 & 0 \\
0 & 1 & 0 \\
0 & -a_{32} & 1
\end{array}
\right) \in (\A)^0 . 
\end{align} 
This shows that 
\begin{align}
[[g]] \ni \varphi g^{\prime} = 
\left( 
\begin{array}{ccc} 
1 & 0 & 0 \\ 
0 & 1 & 0 \\ 
0 & 0 & a_{33} 
\end{array}
\right) . 
\end{align}
Therefore, by putting $\lambda := 1 / a_{33}$, we complete the proof. 
\end{proof}

\subsection{Case of $\g=\mathfrak{r}_{3, a}$ ($-1 \leq a < 1$)} 

In this subsection, we give an explicit expression of $\PM$ for 
$\g=\mathfrak{r}_{3, a}$. 
Throughout this subsection, we fix $a$ satisfying $-1 \leq a < 1$, and 
a basis $\{ e_1 , e_2 , e_3 \}$ of $\mathfrak{r}_{3, a}$ 
whose bracket relations are given by 
\begin{align}
[e_1 , e_2] = e_2, \quad [e_1 , e_3] = a e_3 . 
\end{align} 

From Lemma~\ref{der}, we have 
\begin{align}\label{d2}
\displaystyle \D = \left\{ \left(
\begin{array}{ccc}
x_{11} & 0 & 0 \\
x_{21} & x_{22} & 0 \\
x_{31} & 0 & x_{33} 
\end{array}
\right) 
\mid x_{11}, x_{21}, x_{22}, x_{31}, x_{33} \in \R 
\right\} . 
\end{align}
This yields that 
\begin{align}\label{a2}
\displaystyle (\A)^0 = \left\{ \left(
\begin{array}{ccc}
x_{11} & 0 & 0 \\
x_{21} & x_{22} & 0 \\
x_{31} & 0 & x_{33} 
\end{array}
\right) \mid x_{11}, x_{22}, x_{33} > 0 \right\} . 
\end{align}

\begin{prop}\label{PM2}
Let $\g=\mathfrak{r}_{3, a}$. 
Then the following $U$ is a set of representatives of $\PM$$:$ 
\begin{align}
\displaystyle U = \left\{  \left(
\begin{array}{ccc}
1 & 0 & 0 \\
0 & 1 & 0 \\
0 & \lambda & 1
\end{array}
\right) \mid \lambda \in \mathbb{R} \right\} . 
\end{align}
\end{prop}

\begin{proof}
Take any $g \in \mathrm{GL}_3(\R)$. 
By Lemma~\ref{lem:representatives}, 
we have only to show that there exists $\lambda \in \mathbb{R}$ such that 
\begin{align}
\left(
\begin{array}{ccc}
1 & 0 & 0 \\
0 & 1 & 0 \\
0 & \lambda & 1
\end{array}
\right) \in [[g]] . 
\end{align} 
By (\ref{a2}) and Lemma~\ref{F}, 
there exists $g^{\prime} \in L^{\prime}$ such that $g^{\prime} \in [[g]]$. 
Since $g' \in L^{\prime}$, one can write 
\begin{align}
g^{\prime} = \left(
\begin{array}{ccc}
1 & 0 & 0 \\
0 & 1 & 0 \\
0 & a_{32} & a_{33}
\end{array}
\right), \quad 
a_{33} > 0 . 
\end{align} 
It follows from (\ref{a2}) that 
\begin{align}
\varphi  := \left(
\begin{array}{ccc}
1 & 0 & 0 \\
0 & 1 & 0 \\
0 & 0 & 1/a_{33}
\end{array}
\right) \in (\A)^0 . 
\end{align} 
This yields that 
\begin{align}
[[ g ]] \ni \varphi g^{\prime} = 
\left( 
\begin{array}{ccc} 
1 & 0 & 0 \\ 
0 & 1 & 0 \\ 
0 & a_{32} / a_{33} & 1 
\end{array}
\right) . 
\end{align}
Therefore, by putting $\lambda := a_{32} / a_{33}$, we complete the proof. 
\end{proof}

\subsection{Case of $\g=\mathfrak{r}^{\prime}_{3, a}$ ($a \geq 0$)} 

In this subsection, we give an explicit expression of $\PM$ for 
$\g=\mathfrak{r}^{\prime}_{3, a}$. 
Throughout this subsection, we fix $a$ satisfying $a \geq 0$, 
and a basis $\{ e_1 , e_2 , e_3 \}$ of $\mathfrak{r}_{3, a}$ 
whose bracket relations are given by 
\begin{align}
[e_1 , e_2] = a e_2 - e_3, \quad [e_1 , e_3] = e_2 + a e_3 . 
\end{align} 

From Lemma~\ref{der}, we have 
\begin{align}\label{d3}
\displaystyle \D = \left\{ \left(
\begin{array}{ccc}
x_{11} & 0 & 0 \\ 
x_{21} & x_{22} & - x_{23} \\ 
x_{31} & x_{23} & x_{22} 
\end{array} 
\right) 
\mid x_{11}, x_{21}, x_{22}, x_{23}, x_{31} \in \R 
\right\} . 
\end{align} 
This yields that we can also apply Lemma~\ref{F} for this case. 

\begin{prop}\label{PM3}
Let $\g=\mathfrak{r}^{\prime}_{3, a}$. 
Then the following $U$ is a set of representatives of $\PM$$:$ 
\begin{align}
\displaystyle U = \left\{  \left(
\begin{array}{ccc}
1 & 0 & 0 \\
0 & 1 & 0 \\
0 & 0 & 1/\lambda
\end{array}
\right) \mid \lambda \geq 1 \right\} . 
\end{align}
\end{prop}

\begin{proof}
Take any $g \in \mathrm{GL}_3(\R)$. 
By Lemma~\ref{lem:representatives}, 
we have only to show that there exists $\lambda \geq 1$ such that 
\begin{align}
\left(
\begin{array}{ccc}
1 & 0 & 0 \\
0 & 1 & 0 \\
0 & 0 & 1/\lambda
\end{array}
\right) \in [[g]] . 
\end{align} 
By (\ref{d3}), one can see that $(\A)^0$ contains $F$ defined by (\ref{f}). 
Hence, by Lemma~\ref{F}, 
there exists 
$g^{\prime} \in L^{\prime}$ such that $g^{\prime} \in [[g]]$. 
Since $g' \in L^{\prime}$, one can write 
\begin{align}
g^{\prime} = \left(
\begin{array}{ccc}
1 & 0 & 0 \\
0 & 1 & 0 \\
0 & a_{32} & a_{33}
\end{array}
\right) , \ 
a_{33} > 0 . 
\end{align} 
Then, from (\ref{d3}), one has 
\begin{align} 
R(\theta) := 
\left(
\begin{array}{ccc}
1 & 0 & 0 \\
0 & \cos(\theta) & - \sin(\theta) \\
0 & \sin(\theta) & \cos(\theta)
\end{array}
\right) \in (\A)^0 . 
\end{align} 
It follows from linear algebra (or the theory of Cartan decomposition) that 
\begin{align} 
\mathrm{GL_2}(\mathbb{R}) = 
\mathrm{SO}(2) \cdot 
\left\{ \left(
\begin{array}{cc}
x & 0 \\ 0 & y 
\end{array} 
\right) \mid 
x \geq y > 0 
\right\} 
\cdot \mathrm{O}(2) . 
\end{align} 
This yields that there exist $\theta \in \mathbb{R}$ 
and $k \in \mathrm{O}(3)$ such that 
\begin{align}
[[ g ]] \ni R(\theta) g^{\prime} k = \left(
\begin{array}{ccc}
1 & 0 & 0 \\
0 & x & 0 \\
0 & 0 & y 
\end{array}
\right) =: g'' , \quad 
x \geq y > 0 
\end{align}
By using (\ref{d3}) again, one has 
\begin{align}
\varphi := \left(
\begin{array}{ccc}
1 & 0 & 0 \\
0 & 1/x & 0 \\
0 & 0 & 1/x 
\end{array}
\right) \in (\A)^0 . 
\end{align} 
This yields that 
\begin{align}
[[ g ]] \ni \varphi g'' = 
\left(
\begin{array}{ccc}
1 & 0 & 0 \\
0 & 1 & 0 \\
0 & 0 & y/x 
\end{array}
\right) . 
\end{align}
Therefore, by putting $\lambda := x/y \geq 1$, we complete the proof. 
\end{proof}

\section{Three-dimensional solvsolitons} 

In this section, we give a Milnor-type theorem 
for each three-dimensional solvable Lie algebra $\g$, 
and apply it to determine which points in the moduli space $\PM$ are solvsolitons. 
Note that a classification of three-dimensional solvsolitons has already been obtained by Lauret (\cite{Lau11}), 
but we here reprove it, 
since Milnor-type theorems itself and their application would be interesting. 

\subsection{Preliminaries on curvatures}  

In this subsection, 
we recall the notion of solvsolitons introduced by Lauret (\cite{Lau11}), 
and study the Ricci operators of three-dimensional solvable Lie algebras. 
Note that we discuss everything on a metric Lie algebra $(\g , \naiseki)$, 
instead of the simply-connected Lie group with Lie algebra $\g$ 
equipped with the corresponding left-invariant Riemannian metric. 

\begin{defi}
An inner product $\naiseki$ on a solvable Lie algebra $\g$ is called a 
\textit{solvsoliton} 
if it satisfies 
\begin{align}
\mathrm{Ric}_{\naiseki} \in \mathbb{R} \oplus \Der , 
\end{align}
where $\mathrm{Ric}_{\naiseki}$ is the Ricci operator of $\naiseki$. 
If $\g$ is nilpotent, then a solvsoliton on $\g$ is called a 
\textit{nilsoliton}. 
\end{defi}

Here we recall the definition of the Ricci operator of $(\g , \naiseki)$. 
First of all, the Levi-Civita connection 
$\nabla: \g \times \g \longrightarrow \g$ is given by
\begin{align}
2 \langle \nabla _X Y,Z \rangle = \langle [Z,X],Y \rangle + \langle X,[Z,Y] \rangle +\langle [X,Y], Z \rangle .
\end{align}
The Riemannian curvature $R$ is defined by
\begin{align}
R(X,Y)Z:=\nabla _X \nabla _Y Z - \nabla _Y \nabla _X Z - \nabla _{[X,Y]} Z.
\end{align}
Let $\{ e_i \}$ be an orthonormal basis of $\g$ with respect to $\naiseki$. 
The Ricci operator 
$\mathrm{Ric}_{\naiseki} : \g \rightarrow \g$ is defined by
\begin{align}
\mathrm{Ric}_{\naiseki}(X):=\sum R(X,e_i)e_i.
\end{align}

Let us consider the equivalence relation, 
isometry and scaling in the sense of Definition~\ref{I}. 
Recall that $[\naiseki]$ denotes the equivalence class of $\naiseki$. 
Then it is easy to see the following. 

\begin{prop}
\label{prop:soliton-pm}
Let $\naiseki$ and $\naiseki'$ be inner products on a solvable Lie algebra $\g$, 
and assume that $[\naiseki] = [\naiseki']$. 
If $\naiseki$ is a solvsoliton, then so is $\naiseki^{\prime}$. 
\end{prop}

This proposition is an easy observation, but has an important conclusion. 
That is, it is enough to consider $\PM$ 
to examine whether $\g$ admits a solvsoliton or not. 

\begin{rem}
It is worthwhile to mention that the uniqueness of solvsolitons holds. 
That is, if $\naiseki$ and $\naiseki^{\prime}$ are solvsolitons 
on a solvable Lie algebra $\g$, 
then $[\naiseki] = [\naiseki']$ holds. 
This follows from the proof of \cite[Theorem 5.1]{Lau11}. 
But, we will not use this in the latter arguments. 
In particular, for solvsolitons on three-dimensional solvable Lie algebras, 
the uniqueness can be directly seen 
from our classification. 
\end{rem}

At the end of this subsection, 
we calculate the Ricci curvatures of three-dimensional solvable Lie algebras in a unified way. 

\begin{lem}\label{ric}
Let $\g$ be a three-dimensional solvable Lie algebra, and $\naiseki$ be an inner product on $\g$. 
Suppose that there exist $a,b,c,d \in \mathbb{R}$ and 
an orthonormal basis 
$\{ x_1, x_2, x_3 \}$ 
with respect to $\naiseki$ 
such that the bracket relations are given by 
\begin{align*} 
[x_1, x_2] = a x_2 + b x_3, \quad [x_1, x_3] = c x_2 + d x_3. 
\end{align*} 
Then, the Ricci operator satisfies 
\begin{align*}
\mathrm{Ric}_{\naiseki}(x_i) 
= \left\{ 
\begin{array}{@{\,}ll} 
- (a^2 
+ d^2 
+ (1/2) (b+c)^2 ) \, x_1 & (i = 1) , \\ 
- (
a (a+d) + (1/2) (b^2 - c^2) ) \, x_2 - (a c + b d) \, x_3 & (i = 2) , \\ 
- (a c + b d) \, x_2 - (
d (a+d) - (1/2) (b^2 - c^2) ) \, x_3 & (i = 3) . 
\end{array} 
\right. 
\end{align*}
\end{lem}

\begin{proof}
First of all, we calculate the Levi-Civita connection $\nabla$. 
A direct calculation shows that 
\begin{align} 
\nabla_{x_1} x_1 = 0 , \quad 
\nabla_{x_2} x_2 = a x_1 , \quad 
\nabla_{x_3} x_3 = d x_1 . 
\end{align} 
In order to calculate the other components, 
we use $U : \g \times \g \to \g$ defined by 
\begin{align*}
2 \langle U(X, Y) , Z \rangle = \langle [Z,X] , Y \rangle + \langle X , [Z,Y] \rangle 
\end{align*}
for every $X, Y, Z \in \g$. 
One can easily calculate that 
\begin{align}
U(x_1, x_2) = - (a/2) x_2 - (c/2) x_3 . 
\end{align} 
Note that $U$ is symmetric. 
Hence, one obtains that 
\begin{align} 
\begin{split} 
\nabla_{x_1} x_2 & = (1/2) [x_1 , x_2] + U(x_1, x_2) = ((b-c)/2) x_3 , \\ 
\nabla_{x_2} x_1 & = (1/2) [x_2 , x_1] + U(x_2, x_1) = - a x_2 - ((b+c)/2) x_3 . 
\end{split} 
\end{align} 
By changing the roles of $x_2$ and $x_3$, 
we also have 
\begin{align} 
\nabla_{x_1} x_3 = ((c-b)/2) x_2 , \quad 
\nabla_{x_3} x_1 = - d x_3 - ((b+c)/2) x_2 . 
\end{align} 
A similar calculation shows that $U(x_2, x_3) = ((b+c)/2) x_1$, 
which concludes 
\begin{align} 
\nabla_{x_2} x_3 = ((b+c)/2) x_1 , \quad 
\nabla_{x_3} x_2 = ((b+c)/2) x_1 . 
\end{align} 

One can thus calculate the Riemannian curvatures $R$. 
The above calculations of $\nabla$ yield that 
\begin{align*} 
R(x_1, x_2) x_2 &= - (a^2 + (3/4)b^2 -(1/4)c^2+(1/2)b c) x_1 , \\ 
R(x_1, x_3) x_3 &= - (-(1/4)b^2 +(3/4)c^2+d^2+(1/2)b c) x_1 . 
\end{align*} 
By summing up them, we obtain the Ricci curvature $\mathrm{Ric}_{\naiseki}(x_1)$. 
Similarly, one can obtain 
$\mathrm{Ric}_{\naiseki}(x_2)$ and $\mathrm{Ric}_{\naiseki}(x_3)$ by 
\begin{align*}
R(x_2, x_1) x_1 &= - (a^2 + (3/4)b^2 -(1/4)c^2+(1/2)b c) x_2 - (a c+b d)x_3, \\
R(x_2, x_3) x_3 &= ((1/4)b^2 +(1/4)c^2-a d+(1/2)b c) x_2, \\
R(x_3, x_1) x_1 &= - (a c+b d) x_2 - (-(1/4)b^2 +(3/4)c^2 + d^2 + (1/2)b c)x_3 , \\
R(x_3, x_2) x_2 &= ((1/4)b^2 +(1/4)c^2 -a d +(1/2)b c) x_3 . 
\end{align*} 
This completes the proof of the lemma. 
\end{proof} 

Lemma~\ref{ric} is a slight generalization of some known results. 
In fact, when $a+d \neq 0$ and $ac+bd = 0$, 
the Ricci operators were calculated by Milnor (\cite[Lemma~6.5]{Mil}). 
Note that the Ricci operators are diagonal in this case. 
Ha and Lee (\cite{HL}) also calculated the Ricci operators in some cases, 
which essentially correspond to the case of $a=0$. 
 
\subsection{Preliminaries on Milnor-type theorems}

In this subsection, 
we recall a method for studying all inner products on a given Lie algebra $\g$. 
This method is called a Milnor-type theorem in \cite{HTT}, 
since it generalizes the famous theorem by Milnor (\cite{Mil}). 

\begin{thm}
\label{MTT}
Let $U$ be a set of representatives of $\PM$. 
Then, for every inner product $\naiseki$ on $\g$, 
we have the following$:$ 
\begin{enumerate} 
\item
There exist $h \in U$, $\varphi \in \Aut$, and $k>0$ 
such that 
$\{ \varphi h e_1 , \ldots , \varphi h e_n \}$ 
is an orthonormal basis of $\g$ 
with respect to $k \naiseki$. 
\item 
The matrix expression of $\Der$ with respect to 
$\{ \varphi h e_1 , \ldots , \varphi h e_n \}$ 
coincides with 
\begin{align*}
\{ h^{-1} D h \in \GL \mid D \in \Der \} . 
\end{align*}  
\end{enumerate} 
\end{thm}

\begin{proof} 
The first assertion has been proved in \cite{HTT}. 
We show the second assertion. 
One has that 
$\{ \varphi h e_1 , \ldots , \varphi h e_n \}$ 
and $\{ h e_1 , \ldots , h e_n \}$ 
have the same bracket relations, 
since $\varphi \in \Aut$. 
This yields that the matrix expressions of $\Der$ 
with respect to these two bases are the same. 
Furthermore, 
the latter basis and $\{ e_1 , \ldots , e_n \}$ are related by 
\begin{align} 
(h e_1 , \ldots , h e_n) = (e_1 , \ldots , e_n) h . 
\end{align} 
Therefore, an elementary linear algebra shows that 
the matrix expression of $\Der$ with respect to 
$\{ h e_1 , \ldots , h e_n \}$ coincides with 
the one in the second assertion. 
This completes the proof. 
\end{proof} 

By applying this theorem for a given Lie algebra $\g$, 
we can obtain a Milnor-type theorem. 
More precisely, 
the basis $\{ \varphi h e_1 , \ldots , \varphi h e_n \}$  
plays a similar role to the Milnor frames. 
Note that 
the bracket relations among elements of this basis depend only on $h \in U$, 
since $\varphi$ preserves the bracket product. 

In the following subsections, 
we will study the existence of solvsolitons on three-dimensional solvable Lie algebras. 
Note that we can omit the cases of $\g = \mathfrak{h}_3$ and $\mathfrak{r}_{3, 1}$, because of Remark~\ref{rem-trivial case}. 

\subsection{Case of $\g=\mathfrak{r}_3$}\label{4.4}

In this subsection, 
we prove that $\g = \mathfrak{r}_3$ does not admit solvsolitons. 
The main tool is the following Milnor-type theorem. 

\begin{prop}\label{M1}
For every inner product $\naiseki$ on $\g = \mathfrak{r}_3$, 
there exist $\lambda > 0$, $k > 0$, 
and an orthonormal basis $\{x_1, x_2, x_3\}$ with respect to $k \naiseki$ 
such that the bracket relations are given by 
\begin{align}
[x_1,x_2] = x_2 + \lambda x_3 , \quad [x_1,x_3] = x_3 . 
\end{align} 
Furthermore, the matrix expression of $\Der$ 
with respect to $\{x_1, x_2, x_3\}$ coincides with 
\begin{align*} 
\left\{ \left( 
\begin{array}{ccc} 
0 & 0 & 0 \\ 
x_{21} & x_{22} & 0 \\ 
x_{31} & x_{32} & x_{22} 
\end{array} 
\right) \mid x_{21}, x_{22}, x_{31}, x_{32} \in \R \right\} . 
\end{align*}  
\end{prop}

\begin{proof} 
Let $\{ e_1, e_2, e_3 \}$ be the canonical basis of $\mathfrak{r}_3$. 
Recall that the bracket relations are given by 
\begin{align}
[e_1, e_2] = e_2 + e_3 , \quad [e_1, e_3] = e_3 . 
\end{align}
We have proved in Proposition~\ref{PM1} 
that the following $U$ is a set of representatives of $\PM$: 
\begin{align} 
U :=  
\left\{ 
g_\lambda := 
\left(
\begin{array}{ccc}
1 & 0 & 0 \\
0 & 1 & 0 \\
0 & 0 & 1/\lambda
\end{array}
\right) 
\mid \lambda > 0 \right\} .  
\end{align}
Take any inner product $\naiseki$ on $\g$. 
By Theorem~\ref{MTT}, 
there exist $g_{\lambda} \in U$, $k>0$, and $\varphi \in \Aut$ 
such that 
$\{ \varphi g_{\lambda} e_1,  \varphi g_{\lambda} e_2, \varphi g_{\lambda} e_3 \}$ 
is orthonormal with respect to $k \naiseki$. 
Put $x_i := \varphi g_{\lambda} e_i$ for $i = 1, 2, 3$. 
We calculate 
the bracket relations among them. 
One has 
\begin{align}\label{ge} 
g_\lambda e_1 = e_1 , \quad 
g_\lambda e_2 = e_2 , \quad 
g_\lambda e_3 = (1 / \lambda) e_3 . 
\end{align}
We thus obtain  
\begin{align}
\begin{split}
[g_{\lambda} e_1, g_{\lambda} e_2 ] & = [e_1, e_2] = e_2+e_3 = g_{\lambda} e_2 + \lambda g_{\lambda} e_3, \\
[g_{\lambda} e_1, g_{\lambda} e_3 ] & = [e_1, (1/\lambda )e_3] = (1/\lambda ) e_3 = g_{\lambda} e_3, \\ 
[g_{\lambda} e_2, g_{\lambda} e_3 ] & = [e_2, (1/\lambda )e_3] = 0 . 
\end{split}
\end{align}
Therefore, 
by applying $\varphi \in \Aut$ to the both sides of these equations, 
we obtain 
\begin{align}
\begin{split}
[x_1, x_2] & = [\varphi g_{\lambda} e_1, \varphi g_{\lambda} e_2 ] = \varphi [ g_{\lambda} e_1, g_{\lambda} e_2 ] = x_2 + \lambda x_3, \\
[x_1, x_3] & = [\varphi g_{\lambda} e_1, \varphi g_{\lambda} e_3 ] = \varphi [ g_{\lambda} e_1, g_{\lambda} e_3 ] = x_3, \\
[x_2, x_3] & = [\varphi g_{\lambda} e_2, \varphi g_{\lambda} e_3 ] = \varphi [ g_{\lambda} e_2, g_{\lambda} e_3 ] = 0 . 
\end{split}
\end{align}
This completes the proof of the first assertion. 
We show the second assertion. 
Lemma~\ref{der} yields that, for every $D \in \Der$, 
the matrix expression of $D$ with respect to $\{ e_1, e_2, e_3 \}$ is given by 
\begin{align} 
D = \left( 
\begin{array}{ccc}
0 & 0 & 0 \\
x_{21} & x_{22} & 0 \\
x_{31} & x_{32} & x_{22} 
\end{array}
\right) . 
\end{align}
A direct calculation shows that
\begin{align} 
g_{\lambda}^{-1} \left( 
\begin{array}{ccc}
0 & 0 & 0 \\
x_{21} & x_{22} & 0 \\
x_{31} & x_{32} & x_{22} 
\end{array}
\right) g_{\lambda} 
= \left( 
\begin{array}{ccc}
0 & 0 & 0 \\
x_{21} & x_{22} & 0 \\
\lambda x_{31} & \lambda x_{32} & x_{22} 
\end{array}
\right) . 
\end{align}
Note that $\lambda x_{31}$ and $\lambda x_{32}$ can take any real numbers, 
and are independent of the other components. 
Therefore, by Theorem~\ref{MTT} (2), 
one can obtain the matrix expression of $\Der$ 
with respect to $\{ x_1 , x_2 , x_3 \}$. 
This completes the proof of the second assertion. 
\end{proof}

By applying the Milnor-type theorem, 
Proposition~\ref{M1}, 
we prove that $\mathfrak{r}_3$ does not admit solvsolitons. 

\begin{prop}\label{S1}
The Lie algebra $\g = \mathfrak{r}_3$ does not admit solvsolitons. 
\end{prop}

\begin{proof}
Take any inner product $\naiseki$ on $\g$. 
We show that this is not a solvsoliton. 
By Proposition~\ref{M1}, 
there exist $\lambda > 0$, $k > 0$, 
and an orthonormal basis $\{ x_1, x_2, x_3\}$ with respect to $k \naiseki$ 
such that the bracket relations are given by 
\begin{align}
[x_1,x_2] = x_2 + \lambda x_3 , \quad [x_1,x_3] = x_3 . 
\end{align}
We can assume $k=1$ without loss of generality, 
since solvsolitons are preserved by scaling.  
Then, from Lemma~\ref{ric}, 
the matrix expression of $\mathrm{Ric}_{\naiseki}$ 
with respect to the orthonormal basis $\{ x_1 , x_2 , x_3 \}$ is given by  
\begin{align}
\mathrm{Ric}_{\naiseki} = - \left( 
\begin{array}{ccc}
2+(\lambda^2 /2) & 0 \\ 
0 & 2+(\lambda^2 /2) & \lambda \\
0 & \lambda & 2-(\lambda^2 /2) 
\end{array}
\right) . 
\end{align} 
On the other hand, by Proposition~\ref{M1}, 
one knows the matrix expression of $\Der$ 
with respect to $\{ x_1 , x_2 , x_3 \}$. 
By looking at the $(2,3)$-component, we have 
\begin{align}
\mathrm{Ric}_{\naiseki} \not \in \D . 
\end{align}
This proves that $\naiseki$ is not a solvsoliton. 
\end{proof}

\subsection{Case of $\g=\mathfrak{r}_{3, a}$ ($-1 \leq a < 1$)}\label{4.5}

In this subsection, we classify solvsolitons on $\g = \mathfrak{r}_{3, a}$. 
Throughout this subsection, we fix $a$ satisfying $-1 \leq a < 1$. 
Recall that, for the canonical basis 
$\{ e_1, e_2, e_3 \}$ of $\mathfrak{r}_{3, a}$, 
the bracket relations are given by 
\begin{align}
[e_1, e_2]  = e_2, \quad [e_1, e_3] = a e_3 . 
\end{align}

\begin{prop}\label{M2}
For every inner product $\naiseki$ on $\g = \mathfrak{r}_{3, a}$, 
there exist $\lambda \in \mathbb{R}$, $k > 0$, 
and an orthonormal basis $\{x_1, x_2, x_3\}$ with respect to $k \naiseki$ 
such that the bracket relations are given by 
\begin{align*}
[x_1 , x_2] = x_2 + \lambda (a-1)x_3 , \quad [x_1 , x_3] = a x_3 . 
\end{align*} 
Furthermore, the matrix expression of $\Der$ 
with respect to $\{x_1, x_2, x_3\}$ coincides with 
\begin{align*} 
\left\{ \left( 
\begin{array}{ccc} 
0 & 0 & 0 \\ 
x_{21} & x_{22} & 0 \\ 
x_{31} & \lambda (x_{33} - x_{22}) & x_{33} 
\end{array} 
\right) \mid x_{21}, x_{22}, x_{31}, x_{33} \in \R \right\} . 
\end{align*} 
\end{prop} 

\begin{proof} 
The proof is similar to that of Proposition~\ref{M1}. 
Take any inner product $\naiseki$ on $\mathfrak{r}_{3, a}$. 
By Proposition~\ref{PM2}, the following $U$ is a set of representatives of $\PM$: 
\begin{align}
U := \left\{ g_\lambda := \left(
\begin{array}{ccc}
1 & 0 & 0 \\
0 & 1 & 0 \\
0 & \lambda & 1
\end{array}
\right) \mid \lambda \in \mathbb{R} \right\} . 
\end{align}
By Theorem~\ref{MTT}, 
there exist $g_{\lambda} \in U$, $k>0$, and $\varphi \in \Aut$ 
such that 
\begin{align} 
(x_1 , x_2 , x_3) := (\varphi g_{\lambda} e_1,  \varphi g_{\lambda} e_2, \varphi g_{\lambda} e_3) 
\end{align} 
forms an orthonormal basis with respect to $k \naiseki$. 
We have only to check the bracket relations. 
By definition, we have 
\begin{align}
g_\lambda e_1 = e_1 , \quad 
g_\lambda e_2 = e_2 + \lambda e_3 , \quad 
g_\lambda e_3 = e_3 . 
\end{align}
One can thus calculate that 
\begin{align}
\begin{split}
[g_{\lambda} e_1, g_{\lambda} e_2 ] 
& = [e_1, e_2 + \lambda e_3] 
= e_2+ a \lambda e_3 
= (g_{\lambda} e_2 - \lambda g_{\lambda} e_3) + a \lambda g_{\lambda} e_3 \\ 
&= g_{\lambda} e_2 + \lambda (a-1) e_3, \\ 
[g_{\lambda} e_1, g_{\lambda} e_3 ] 
& = [e_1, e_3] = a e_3 = a g_{\lambda} e_3, \\ 
[g_{\lambda} e_2, g_{\lambda} e_3 ] 
& = [e_2 + \lambda e_3, e_3] = 0 . 
\end{split}
\end{align}
By applying $\varphi \in \Aut$, one completes the proof of the first assertion. 
The second assertion follows from Lemma~\ref{der} and Theorem~\ref{MTT}. 
In fact, one has 
\begin{align}
g_\lambda^{-1} \left(
\begin{array}{ccc}
0 & 0 & 0 \\
x_{21} & x_{22} & 0 \\
x_{31} & 0 & x_{33}
\end{array}
\right) g_\lambda = \left( 
\begin{array}{ccc} 
0 & 0 & 0 \\ 
x_{21} & x_{22} & 0 \\ 
- \lambda x_{21} + x_{31} & \lambda (x_{33} - x_{22}) & x_{33} 
\end{array} 
\right) . 
\end{align}
This completes the proof, 
since $- \lambda x_{21} + x_{31}$ can take any real number and 
is independent of the other components. 
\end{proof}

By applying the Milnor-type theorem, 
Proposition~\ref{M2}, 
one can classify solvsolitons on $\g = \mathfrak{r}_{3, a}$. 
Recall that $\naiseki_0$ is the inner product on $\g$ so that 
the canonical basis $\{ e_1, e_2, e_3 \}$ is orthonormal.

\begin{prop}\label{S2}
An inner product $\naiseki$ on $\g=\mathfrak{r}_{3, a}$ 
is a solvsoliton if and only if $[\naiseki] = [\naiseki_0]$.
\end{prop}

\begin{proof} 
First of all, we show the ``if''-part. 
We have only to show that $\naiseki_0$ is a solvsoliton. 
By Lemma~\ref{ric}, one knows 
\begin{align} 
\mathrm{Ric}_{\naiseki_0} = - \left( 
\begin{array}{ccc}
1+a^2 & 0 & 0 \\ 
0 & 1+a & 0 \\
0 & 0 & a(1+a) 
\end{array}
\right) , 
\end{align}
One also knows by Lemma~\ref{der} that 
\begin{align}
\D = \left\{ \left(
\begin{array}{ccc}
x_{11} & 0 & 0 \\
x_{21} & x_{22} & 0 \\
x_{31} & 0 & x_{33}
\end{array} 
\right) 
\mid x_{11}, x_{21}, x_{22}, x_{31}, x_{33} \in \R 
\right\} . 
\end{align}
Then we have $\mathrm{Ric}_{\naiseki_0} \in \D$, that is, 
$\naiseki_0$ is a solvsoliton.  

We show the ``only if''-part. 
Take any inner product $\naiseki$ 
on $\g=\mathfrak{r}_{3, a}$, and assume that it is a solvsoliton. 
Proposition~\ref{M2} yields that 
there exist $\lambda \in \mathbb{R}$, $k > 0$, 
and an orthonormal basis $\{x_1, x_2, x_3\}$ with respect to $k \naiseki$ 
such that the bracket relations are given by 
\begin{align}
[x_1 , x_2] = x_2 + \lambda (a-1)x_3 , \quad [x_1 , x_3] = a x_3 . 
\end{align}
We can assume $k = 1$ without loss of generality. 
Hence $\{ x_1, x_2, x_3 \}$ is orthonormal. 
For simplicity of the notation, we put 
\begin{align} 
T := (1/2) \lambda^2 (a-1)^2 . 
\end{align} 
Then, from Lemma~\ref{ric}, one obtains the matrix expressions of $\mathrm{Ric}_{\naiseki}$ 
with respect to the basis $\{ x_1, x_2, x_3 \}$ as follows: 
\begin{align}\label{R2}
\mathrm{Ric}_{\naiseki} = - \left( 
\begin{array}{ccc}
1+a^2+T & 0 & 0 \\ 
0 & 1+a+T & \lambda a(a-1) \\
0 & \lambda a(a-1) & a+a^2-T 
\end{array}
\right) .  
\end{align}
On the other hand, 
Proposition~\ref{M2} gives the matrix expression 
with respect to $\{ x_1 , x_2 , x_3 \}$ as follows: 
\begin{align}
\label{eq:R+Der2}
\D = 
\left\{ \left(
\begin{array}{ccc}
x_{11} & 0 & 0 \\
x_{21} & x_{22} & 0 \\
x_{31} & \lambda (x_{33} - x_{22}) & x_{33} 
\end{array}
\right) \right\} . 
\end{align} 

We here claim that $\lambda = 0$. 
Recall that $\naiseki$ is a solvsoliton. Hence, by looking at the $(2,3)$-component, 
we have 
\begin{align}
\lambda a (a-1) = 0 . 
\end{align}
Assume that $\lambda \neq 0$. 
Since $-1 \leq a < 1$, one has $a = 0$. 
Then, by looking at the $(3,2)$-component, we have 
\begin{align}
0 = \lambda ( -T - (1+T))
= \lambda (-1 - \lambda^2) \neq 0 . 
\end{align}
This is a contradiction, which shows the claim. 
 
Since $\lambda = 0$, one can see that 
$\{ e_1 , e_2 , e_3 \}$ and $\{ x_1 , x_2 , x_3 \}$ have the same bracket relations. 
Thus, a linear map $F : \g \to \g$ satisfying 
\begin{align}
F(e_i) = x_i \quad (i = 1,2,3) 
\end{align}
gives an isometry from $(\g , \naiseki_0)$ onto $(\g , \naiseki)$. 
This proves $[\naiseki] = [\naiseki_0]$. 
\end{proof}

\subsection{Case of $\g=\mathfrak{r}^{\prime}_{3, a}$ ($a \geq 0$)}\label{4.6}

In this subsection, we classify solvsolitons on 
$\g = \mathfrak{r}^{\prime}_{3, a}$. 
Throughout this subsection, we fix $a$ satisfying $a \geq 0$. 
Recall that, for the canonical basis $\{ e_1, e_2, e_3 \}$, 
the bracket relations are given by 
\begin{align}
[e_1, e_2] = a e_2 - e_3 , \quad [e_1, e_3] = e_2 + a e_3 . 
\end{align}

\begin{prop}\label{M3}
For every inner product $\naiseki$ on $\g=\mathfrak{r}^{\prime}_{3, a}$, 
there exist $\lambda \geq 1$, $k > 0$, 
and an orthonormal basis $\{x_1, x_2, x_3\}$ with respect to $k \naiseki$ 
such that the bracket relations are given by 
\begin{align}
[x_1,x_2] = a x_2 - \lambda x_3 , \quad [x_1,x_3] = (1/\lambda) x_2 + a x_3 . 
\end{align} 
Furthermore, the matrix expression of $\Der$ 
with respect to $\{x_1, x_2, x_3\}$ coincides with 
\begin{align*} 
\left\{ \left( 
\begin{array}{ccc} 
0 & 0 & 0 \\
x_{21} & x_{22} & x_{23} \\ 
x_{31} &  -\lambda^2 x_{23} & x_{22} 
\end{array} 
\right) 
\mid x_{21}, x_{22}, x_{23}, x_{31} \in \R  
\right\} . 
\end{align*}  
\end{prop}

\begin{proof}
The proof is similar to that of Proposition~\ref{M1}. 
Take any inner product $\naiseki$ on $\mathfrak{r}^{\prime}_{3, a}$. 
By Proposition~\ref{PM3}, the following $U$ is a set of representatives of $\PM$: 
\begin{align}
U := \left\{ g_\lambda := \left(
\begin{array}{ccc}
1 & 0 & 0 \\
0 & 1 & 0 \\
0 & 0 & 1/\lambda
\end{array}
\right) \mid \lambda \geq 1 \right\} . 
\end{align}
By Theorem~\ref{MTT}, 
there exist $g_{\lambda} \in U$, $k>0$, and $\varphi \in \Aut$ 
such that 
\begin{align} 
(x_1 , x_2 , x_3) := (\varphi g_{\lambda} e_1,  \varphi g_{\lambda} e_2, \varphi g_{\lambda} e_3) 
\end{align} 
forms an orthonormal basis with respect to $k \naiseki$. 
We have only to check the bracket relations. 
By definition, we have 
\begin{align}
g_\lambda e_1 = e_1 , \quad 
g_\lambda e_2 = e_2 , \quad 
g_\lambda e_3 = (1/\lambda)e_3 . 
\end{align}
One can thus calculate that 
\begin{align}
\begin{split}
[g_{\lambda} e_1, g_{\lambda} e_2 ] 
& = [e_1, e_2] 
= a e_2 - e_3 
= a g_{\lambda} e_2 - \lambda g_{\lambda} e_3, \\ 
[g_{\lambda} e_1, g_{\lambda} e_3 ] & = [e_1, (1/\lambda )e_3] = (1/\lambda ) (e_2 + a e_3) = (1/\lambda ) g_{\lambda} e_2 + a g_{\lambda} e_3, \\ 
[g_{\lambda} e_2, g_{\lambda} e_3 ] & = [e_2, (1/\lambda )e_3] = 0.
\end{split}
\end{align}
By applying $\varphi \in \Aut$, one completes the proof of the first assertion. 
The second assertion follows from Lemma~\ref{der} and Theorem~\ref{MTT}. 
In fact, one has 
\begin{align}
g_\lambda^{-1} \left(
\begin{array}{ccc}
0 & 0 & 0 \\
x_{21} & x_{22} & -x_{23} \\
x_{31} & x_{23} & x_{22}
\end{array}
\right) g_\lambda = \left( 
\begin{array}{ccc} 
0 & 0 & 0 \\ 
x_{21} & x_{22} & -(1/\lambda)x_{23} \\ 
\lambda x_{31} & \lambda x_{23} & x_{22} 
\end{array} 
\right) . 
\end{align}
This completes the proof by changing $\lambda x_{31}$ to $x_{31}$, 
and $-(1/\lambda)x_{23}$ to $x_{23}$. 
\end{proof} 

By applying the Milnor-type theorem, Proposition~\ref{M3}, 
one can classify solvsolitons on 
$\g = \mathfrak{r}^{\prime}_{3, a}$.  
In fact, this admits a left-invariant Einstein metric. 
Recall that $\naiseki_0$ is the inner product 
so that the canonical basis $\{ e_1, e_2, e_3 \}$ is orthonormal.

\begin{prop}\label{S3} 
An inner product $\naiseki$ on 
$\g =  \mathfrak{r}^{\prime}_{3, a}$ is a solvsoliton 
if and only if $[\naiseki] = [\naiseki_0]$. 
In fact, $\naiseki_0$ is Einstein. 
\end{prop}

\begin{proof}
The proof is similar to that of Proposition~\ref{S2}.
First of all, we show the ``if''-part. 
By Lemma~\ref{ric}, one knows 
\begin{align} 
\mathrm{Ric}_{\naiseki_0} = - \left( 
\begin{array}{ccc}
2 a^2 & 0 & 0 \\ 
0 & 2 a^2 & 0 \\
0 & 0 & 2 a^2 
\end{array}
\right) . 
\end{align}
This shows that $\naiseki_0$ is 
Einstein, and hence a solvsoliton. 
 
We show the ``only if''-part. 
Take any inner product $\naiseki$ 
on $\g=\mathfrak{r}^{\prime}_{3, a}$, and assume that it is a solvsoliton. 
Proposition~\ref{M3} yields that 
there exist $\lambda \geq 1$, $k > 0$, 
and an orthonormal basis $\{x_1, x_2, x_3\}$ with respect to $k \naiseki$ 
such that the bracket relations are given by 
\begin{align}
[x_1 , x_2] = a x_2 - \lambda x_3 , \quad [x_1 , x_3] = (1/\lambda) x_2 + a x_3 . 
\end{align}
We can assume $k = 1$ without loss of generality. 
Hence $\{ x_1, x_2, x_3 \}$ is orthonormal. 
For simplicity of the notation, we put 
\begin{align} 
S := \lambda -(1/\lambda ) . 
\end{align} 
Then, from Lemma~\ref{ric}, one obtains the matrix expressions of $\mathrm{Ric}_{\naiseki}$ 
with respect to the basis $\{ x_1, x_2, x_3 \}$ as follows: 
\begin{align}\label{R3}
\mathrm{Ric}_{\naiseki}=- \frac{1}{2}\left(
\begin{array}{ccc}
4 a^2+S^2 & 0 & 0 \\ 
0 & 4 a^2 + (\lambda^2 -(1/\lambda )^2) 
& -2a S \\
0 & -2a S & 4 a^2- (\lambda^2 -(1/\lambda )^2) 
\end{array}
\right) . 
\end{align}
On the other hand, 
Proposition~\ref{M3} gives the matrix expression 
with respect to $\{ x_1 , x_2 , x_3 \}$ as follows: 
\begin{align}
\label{eq:R+Der3}
\D = 
\left\{ \left(
\begin{array}{ccc}
x_{11} & 0 & 0 \\
x_{21} & x_{22} & x_{23} \\
x_{31} & -\lambda^2 x_{23} & x_{22} 
\end{array}
\right) \right\} . 
\end{align} 

We here show that $\lambda = 1$. 
Recall that $\naiseki$ is a solvsoliton. Hence, By looking at the $(2,2)$ and $(3,3)$-components, we have 
\begin{align}
4 a^2 + (\lambda^2 -(1/\lambda )^2) = 4 a^2 - (\lambda^2 -(1/\lambda )^2) . 
\end{align}
Since $\lambda \geq 1$, this yields that 
\begin{align}
\lambda = 1 . 
\end{align}
 
Since $\lambda = 1$, one can see that 
$\{ e_1 , e_2 , e_3 \}$ and $\{ x_1 , x_2 , x_3 \}$ have the same bracket relations. 
Thus, a linear map $F : \g \to \g$ satisfying 
\begin{align}
F(e_i) = x_i \quad (i = 1,2,3) 
\end{align}
gives an isometry from $(\g , \naiseki_0)$ onto $(\g , \naiseki)$. 
This proves $[\naiseki] = [\naiseki_0]$. 
\end{proof}

\section{The minimality of the corresponding submanifolds}

In this section, 
we study the actions of $\A$ and examine the minimality of its orbits, 
the corresponding submanifolds to left-invariant metrics. 
After some necessary preliminaries in Subsection~5.1, 
we study the cases of $\g = \mathfrak{r}_3$, 
$\mathfrak{r}_{3, a}$ ($-1 \leq a < 1$), 
and $\mathfrak{r}_{3, a}^{\prime}$ ($a \geq 0$)
in Subsections~5.2, 5.3, and 5.4, respectively. 
We have only to study these cases, 
since the actions of $\A$ is transitive for the remaining cases 
$\g=\h$ and $\mathfrak{r}_{3, 1}$. 

\subsection{Preliminary}

In this subsection, 
we review some of the standard facts on reductive homogeneous spaces 
and homogeneous submanifolds. 
We refer to \cite{And, Bes}. 

Let $U/K$ be a reductive homogeneous space with a reductive decomposition 
\begin{align}
\mathfrak{u} = \mathfrak{k} \oplus \mathfrak{m} . 
\end{align}
As in Subsection~2.1, 
denote by $\pi : U \rightarrow U/K$ the natural projection, 
and by $o := \pi(e)$ the origin of $U/K$. 
We identify $\m$ with the tangent space $\mathrm{T}_o (U/K)$ at $o$ by 
\begin{align}
\mathrm{d}\pi_e |_{\m} : \m \rightarrow \mathrm{T}_o (U/K) . 
\end{align}
In the following, we equip a $U$-invariant Riemannian metric $g$ on $U/K$. 

We here recall a formula for the Levi-Civita connection $\nabla$ of $g$. 
For any $X \in \mathfrak{u}$, we define the fundamental vector field 
$X^\ast$ on $U/K$ by 
\begin{align}
X_p^\ast=\frac{d}{dt}({\rm exp}tX).p|_{t=0} 
\quad (
\mbox{for $p \in U/K$}) . 
\end{align} 
Let $X, Y, Z \in \mathfrak{u}$. 
Then one knows 
\begin{align}\label{d}
X^\ast_o & = \mathrm{d}\pi_e (X) , \\ 
\label{seki}
[X^\ast, Y^\ast] & = -[X, Y]^\ast , \\ 
\label{conn}
2 g( \nabla _{X^\ast} Y^\ast, Z^\ast ) & = 
g( [X^\ast, Y^\ast], Z^\ast ) + g( [X^\ast, Z^\ast], Y^\ast ) + g( X^\ast, [Y^\ast, Z^\ast] ) . 
\end{align}

We now consider homogeneous submanifolds in $(U/K , g)$. 
Let $U^{\prime}$ be a Lie subgroup of $U$, 
and consider the orbit $U^\prime.o$ through the origin $o$. 
Let $\mathfrak{u}^{\prime}$ be the Lie algebra of $U^{\prime}$, 
and denote by $\naiseki$ the inner product on $\m$ corresponding to $g$.
We define
\begin{align}
\m^{\prime} := \mathrm{d} \pi_e (\mathfrak{u}^{\prime}) 
\cong \mathrm{T}_o (U^\prime.o) . 
\end{align}
Denote by 
$\m \ominus \m^{\prime}$ the orthogonal complement of $\m^{\prime}$ 
in $\m$ with respect to $\naiseki$. 
Then, the second fundamental form 
$h : \m^{\prime} \times \m^{\prime} \rightarrow \m \ominus \m^{\prime}$ of 
$U^{\prime}.o$ at $o$ is defined by 
\begin{align}
h(X^\ast_o, Y^\ast_o) := (\nabla _{X^\ast} Y^\ast - \nabla^{\prime} _{X^\ast} Y^\ast)_o 
\quad 
(\mbox{for $X,Y \in \mathfrak{u}^{\prime}$}),
\end{align}
where $\nabla^{\prime}$ is the Levi-Civita connection of $U^\prime.o$ 
with respect to the induced metric. 
Take $Z \in \mathfrak{u}$ satisfying 
$Z^\ast_o \in \m \ominus \m^{\prime}$. 
From (\ref{seki}) and (\ref{conn}), one obtains 
\begin{align}\label{sec}
\begin{split}
2 \langle  h(X^\ast_o, Y^\ast_o), Z^\ast_o \rangle 
&= \langle [Z, X]^\ast_o, Y^\ast_o \rangle + \langle X^\ast_o, [Z, Y]^\ast_o \rangle .
\end{split}
\end{align}
The mean curvature vector of $U^{\prime}.o$ at o is defined by 
\begin{align}\label{mcv}
H := - (1/k) \mathrm{tr} (h) = - (1/k)
\textstyle 
\sum h(E_i^{\prime}, E_i^{\prime}) , 
\end{align}
where $\{ E_i^{\prime} \}$ is an orthonormal basis of $\mathfrak{m}^{\prime}$, 
and $k$ is the dimension of $U^{\prime}.o$. 
We call $U^{\prime}.o$ \textit{minimal} 
if its mean curvature vector is equal to zero. 

In the following subsections, 
we will calculate the mean curvature vectors 
of the corresponding submanifolds in $\mathrm{GL}_3(\mathbb{R})/\mathrm{O}(3)$ 
with respect to the natural Riemannian metric (see Section~2). 
We will frequently use 
\begin{align}
\mathrm{d}\pi_e : \mathfrak{gl}_3(\mathrm{R}) \rightarrow \mathrm{sym}(3) : 
X \mapsto (1/2)(X+ {}^t X  ) . 
\end{align}

\subsection{Case of $\g=\mathfrak{r}_3$}

In this subsection, we study the case of $\g = \mathfrak{r}_3$. 
First of all, by direct calculations, 
one has 
\begin{align}
\label{aut1}
\Aut = \left\{ \left(
\begin{array}{ccc}
1 & 0 & 0 \\
x_{21} & x_{22} & 0 \\
x_{31} & x_{32} & x_{22} 
\end{array}
\right) \mid x_{22} \neq 0 \right\} . 
\end{align}
This easily yields that 
\begin{align}
\label{eq:aut1}
\A = \left\{ \left(
\begin{array}{ccc}
x_{11} & 0 & 0 \\
x_{21} & x_{22} & 0 \\
x_{31} & x_{32} & x_{22} 
\end{array}
\right) \mid x_{11} , x_{22} \neq 0 \right\} . 
\end{align} 
From Proposition \ref{PM1}, the expression of $\PM$ is given as follows: 
\begin{align}\label{epm1}
\displaystyle \PM = \left\{ [ g_{\lambda}.\naiseki _0 ] \mid g_{\lambda} = \left(
\begin{array}{ccc}
1 & 0 & 0 \\
0 & 1 & 0 \\
0 & 0 & 1/\lambda 
\end{array}
\right) , \ \lambda > 0 \right\} . 
\end{align} 
For any $\lambda > 0$, one can see that 
\begin{align}
\label{eq:normalize}
g_{\lambda}^{-1} 
(\A) 
g_{\lambda}=\A.
\end{align}
This is an easy observation, 
but very important to get the following lemma. 

\begin{lem}\label{action1}
Let $\g=\mathfrak{r}_3$. Then the action of $\A$ is of cohomogeneity one, 
and all orbits are isometrically congruent to each other. 
\end{lem}

\begin{proof}
In order to prove the action of $\A$ is of cohomogeneity one, 
it is enough to show that the orbit through $\naiseki_0$ is of codimension one. 
From (\ref{eq:aut1}), 
it is easy to see that
\begin{align}
\dim \A = 5 , \quad 
\dim (\A \cap \mathrm{O}(3)) = 0. 
\end{align}
Therefore $\A.\naiseki_0$ has dimension $5$. 
This completes the proof, 
since the ambient space $\mathrm{GL}_3(\mathbb{R})/\mathrm{O}(3)$ 
has dimension $6$. 

Next we prove that all orbits are isometrically congruent to each other. 
Take any $\naiseki$ and $\naiseki^\prime$. 
By Proposition~\ref{PM1}, there exist $\lambda, \lambda^{\prime} > 0$ such that 
\begin{align}
\A.\naiseki = \A.(g_{\lambda}.\naiseki_0) ,\\
\A.\naiseki^{\prime} = \A.(g_{\lambda^{\prime}}.\naiseki_0) . 
\end{align}
We put $\mu := \lambda^{\prime} / \lambda >0$, 
and take $g_{\mu} \in \mathrm{GL}_3(\mathbb{R})$. 
Then (\ref{eq:normalize}) yields that 
\begin{align}
\begin{split}
g_{\mu} \A.\naiseki &= g_{\mu} \A.(g_\lambda.\naiseki_0) \\
& = g_{\mu} (g_{\mu}^{-1} \A g_{\mu}).(g_\lambda.\naiseki_0) \\
& = \A.(g_{\lambda^{\prime}}.\naiseki_0) \\
& = \A.\naiseki^{\prime} . 
\end{split}
\end{align}
Thus $g_\mu$ maps the first orbit onto the second one, 
which completes the proof.
\end{proof}

We refer to \cite{K-T} for actions all of whose orbits are 
isometrically congruent to each other. 
Our idea of the proof of Lemma~\ref{action1} 
comes from the arguments in \cite{K-T}. 

\begin{prop}
Let $\g = \mathfrak{r}_3$.  
Then the action of $\A$ 
on $\mathrm{GL}_3(\mathbb{R})/\mathrm{O}(3)$ 
has no minimal orbits. 
\end{prop}

\begin{proof}
Consider the action of $\A$ on $\mathrm{GL}_3(\mathbb{R})/\mathrm{O}(3)$. 
From Lemma \ref{action1}, 
all orbits are isometrically congruent to each other. 
Thus it is sufficient to prove that the orbit through the origin $\naiseki_0$ is not minimal. 
We calculate the mean curvature vector of $\A.\naiseki_0$. 
One can see from (\ref{d1}) that
\begin{align}
\mathfrak{u}^{\prime} & := \D 
= \left\{ \left(
\begin{array}{ccc}
x_{11} & 0 & 0 \\
x_{21} & x_{22} & 0 \\
x_{31} & x_{32} & x_{22} 
\end{array}
\right) \right\} , \\ 
\m^\prime & := 
\mathrm{d}\pi_e (\mathfrak{u}^{\prime})
= 
\left\{ 
\left(
\begin{array}{ccc}
x_{11} & x_{21} & x_{31} \\
x_{21} & x_{22} & x_{32} \\
x_{31} & x_{32} & x_{22} 
\end{array}
\right) \right\} . 
\end{align}
Let us denote by $E_{ij}$ the matrix whose $(i,j)$-entry is $1$ and others are $0$.
We define a basis $\{ X_1 , \ldots , X_5 \}$ of $\mathfrak{u}^\prime$ by 
\begin{align}
\begin{split}
&
X_1 := E_{11}, \quad 
X_2 := (1/\sqrt{2}) (E_{22}+E_{33}) , \\ 
&
X_3 := \sqrt{2} E_{21} , \quad 
X_4 := \sqrt{2} E_{31} , \quad 
X_5 := \sqrt{2} E_{32} . 
\end{split}
\end{align}
Furthermore we put 
\begin{align}
X_i^{\prime} := (X_i)_o^\ast = (1/2) (X_i + {}^t X_i) , 
\quad 
A := (1/\sqrt{2})(E_{22}-E_{33}). 
\end{align}  
Then $\{ X_1^{\prime} , \ldots , X_5^{\prime} \}$ is an orthonormal basis of $\m^{\prime}$, 
and $\{ A \}$ is an orthonormal basis of $\m \ominus \m^{\prime}$. 
Recall that the mean curvature vector $H$ is given by 
\begin{align}
H = -(1/5) 
\textstyle \sum h(X_i^\prime , X_i^\prime) , \quad 
\langle h(X_i^\prime , X_i^\prime) , A \rangle 
= \langle [A , X_i]_o^\ast , (X_i)_o^\ast \rangle . 
\end{align}
The bracket products $[A , X_i]$ satisfy 
\begin{align}
[A , X_1] = [A , X_2] = 0 , \ 
[A , X_3] = E_{21} , \ 
[A , X_4] = - E_{31} , \ 
[A , X_5] = -2 E_{32} . 
\end{align}
Therefore, one has 
\begin{align}
\begin{split}
\langle [A , X_3]_o^\ast , (X_3)_o^\ast \rangle 
& 
= \langle (1/2) (E_{21}+E_{12}) , (\sqrt{2}/2) (E_{21}+E_{12}) \rangle 
= \sqrt{2} / 2 , \\ 
\langle [A , X_4]_o^\ast , (X_4)_o^\ast \rangle 
& 
= \langle (1/2) (-E_{31}-E_{13}) , (\sqrt{2}/2) (E_{31}+E_{13}) \rangle 
= - \sqrt{2} / 2 , \\ 
\langle [A , X_5]_o^\ast , (X_5)_o^\ast \rangle 
& 
= \langle (-E_{32}-E_{23}) , (\sqrt{2}/2) (E_{32}+E_{23}) \rangle 
= - \sqrt{2} . 
\end{split}
\end{align}
This yields that 
\begin{align}
H=(\sqrt{2}/5)A \neq 0. 
\end{align}
Therefore, $\A.\naiseki_0$ is not minimal, 
which completes the proof. 
\end{proof}

\subsection{Case of $\g=\mathfrak{r}_{3,a}$ ($-1 \leq a < 1$)}

In this subsection, 
we study the case of $\g=\mathfrak{r}_{3,a}$. 
Throughout this subsection, we fix $a$ satisfying $-1 \leq a < 1$. 
Recall that, from Lemma \ref{der}, 
one has 
\begin{align}\label{de2}
\displaystyle \D = \left\{ \left(
\begin{array}{ccc}
x_{11} & 0 & 0 \\
x_{21} & x_{22} & 0 \\
x_{31} & 0 & x_{33} 
\end{array}
\right) 
\mid x_{11}, x_{21}, x_{22}, x_{31}, x_{33} \in \R 
\right\} . 
\end{align}
The expression of $\PM$ is given in Proposition~\ref{PM2} as follows: 
\begin{align}
\displaystyle \PM = \left\{ [ g_{\lambda}.\naiseki _0 ] \mid g_{\lambda} = \left(
\begin{array}{ccc}
1 & 0 & 0 \\
0 & 1 & 0 \\
0 & \lambda & 1
\end{array}
\right) , \ \lambda \in \mathbb{R} \right\} . 
\end{align}

\begin{prop}\label{Pr2}
Let $\g=\mathfrak{r}_{3, a}$. 
Then, we have
\begin{enumerate}
\item The action of $\A$ is of cohomogeneity one, and all orbits are hypersurfaces. 
\item $\A.\naiseki_0$ is the unique minimal orbit. 
\end{enumerate}
\end{prop}

\begin{proof}
Take any $\naiseki$. 
In order to prove (1), 
we show that $\A.\naiseki$ is a hypersurface, 
that is, has dimension $5$. 
From the expression of $\PM$, 
there exists $\lambda \in \mathbb{R}$ such that 
\begin{align}
\A.\naiseki = \A.(g_{\lambda}. \naiseki_0) . 
\end{align}
Let us define 
\begin{align}
U^{\prime} := g_{\lambda}^{-1} (\A) g_{\lambda} . 
\end{align}
Then, since $g_{\lambda}^{-1}$ gives an isometry, 
one has an isometric congruence 
\begin{align}\label{Orb2} 
\A.(g_{\lambda}. \naiseki_0) \cong U^{\prime}.\naiseki_0 . 
\end{align}
Let $\mathfrak{u}^{\prime}$ be the Lie algebra of $U^{\prime}$. 
From the expression of $\D$, one can directly calculate 
\begin{align}\label{LAO2}
\mathfrak{u}^{\prime} 
= g_{\lambda}^{-1} (\D) g_{\lambda} 
= \left\{ \left(
\begin{array}{ccc}
x_{11} & 0 & 0 \\
x_{21} & x_{22} & 0 \\
x_{31} & -\lambda (x_{22}-x_{33}) & x_{33} 
\end{array}
\right) \right\} .
\end{align}
Thus it is easy to check that 
\begin{align}
\dim \mathfrak{u}^{\prime} = 5 , \quad 
\dim (\mathfrak{u}^{\prime} \cap \mathfrak{o}(3)) = 0 . 
\end{align}
Therefore $U^{\prime}.\naiseki_0$ has dimension $5$, 
which completes the proof of (1). 

In order to prove (2), 
we have only to show that 
$U^{\prime}.\naiseki$ is minimal if and only if $\lambda = 0$. 
From (\ref{LAO2}), one can see that 
\begin{align}
\m^{\prime} & := \mathrm{d}\pi_e (\mathfrak{u}^{\prime})
= 
\left\{ 
\left(
\begin{array}{ccc}
x_{11} & x_{21} & x_{31} \\
x_{21} & x_{22} & (-\lambda /2) (x_{22}-x_{33}) \\
x_{31} & (-\lambda /2) (x_{22}-x_{33}) & x_{33} 
\end{array}
\right) 
\right\} . 
\end{align} 
We define a basis $\{ X_1, \ldots, X_5 \}$ of $\mathfrak{u}^{\prime}$ by 
\begin{align}
\begin{split}
&
X_1 := E_{11}, \quad 
X_2 := (1 / \sqrt{2}) (E_{22}+ E_{33}) , \\ 
&
X_3 := (1 / \sqrt{2(1+\lambda^2)}) \left( E_{22}-E_{33}-2\lambda E_{32} \right) , \\ 
& 
X_4 := \sqrt{2} E_{21} , \quad 
X_5 := \sqrt{2} E_{31} . 
\end{split}
\end{align}
Let us put 
\begin{align}
X_i^{\prime} := (X_i)_o^\ast = (1/2) (X_i + {}^t X_i) . 
\end{align} 
Then $\{ X_1^{\prime}, \ldots , X_5^{\prime} \}$ 
is an orthonormal basis of $\m^{\prime}$. 
Furthermore, define
\begin{align}
A &:= 
(1 / \sqrt{2(1+\lambda^2)}) \left( -\lambda E_{22}+\lambda E_{33}-2 E_{32} \right), 
\quad 
A^{\prime} := (A)_o^\ast . 
\end{align}  
Then 
$\{ A^{\prime} \}$ is an orthonormal basis of $\m \ominus \m^{\prime}$. 
Recall that the mean curvature vector $H$ is given by 
\begin{align}
H = -(1/5) 
\textstyle \sum h(X_i^\prime , X_i^\prime) , \quad 
\langle h(X_i^\prime , X_i^\prime) , A^{\prime} \rangle 
= \langle [A , X_i]_o^\ast , (X_i)_o^\ast \rangle . 
\end{align}
The bracket products $[A , X_i]$ satisfy 
\begin{align}
\begin{split}
&
[A , X_1] = [A , X_2] = 0, \quad [A , X_3] = -2 E_{32} , \\ 
&
[A , X_4] = - (1 / \sqrt{1 + \lambda^2}) (\lambda E_{21} + 2 E_{31})  , \quad 
[A , X_5] = (\lambda / \sqrt{1+\lambda^2}) E_{31}. 
\end{split}
\end{align}
Hence, one has 
\begin{align}
\begin{split}
\langle [A , X_3]_o^\ast , (X_3)_o^\ast \rangle 
&  = 2 \lambda / \sqrt{2 (1+\lambda^2)} , \\ 
\langle [A , X_4]_o^\ast , (X_4)_o^\ast \rangle  
& = - \lambda / \sqrt{2 (1+\lambda^2)} , \\ 
\langle [A , X_5]_o^\ast , (X_5)_o^\ast \rangle  
& = \lambda / \sqrt{2 (1+\lambda^2)} . 
\end{split}
\end{align}
This yields that 
\begin{align}
H = - 2 \lambda / (5 \sqrt{2 (1+\lambda^2)}) A^{\prime} . 
\end{align}
Therefore, $H=0$ if and only if $\lambda = 0$. 
This completes the proof of (2). 
\end{proof}

\subsection{Case of $\g=\mathfrak{r}^{\prime}_{3,a}$ ($a \geq 0$)}

In this subsection, 
we study the case of $\g=\mathfrak{r}^{\prime}_{3,a}$. 
Throughout this subsection, we fix $a$ satisfying $a \geq 0$. 
Recall that, from Lemma~\ref{der}, $\D$ is given by
\begin{align}\label{de3}
\displaystyle \D = \left\{ \left(
\begin{array}{ccc}
x_{11} & 0 & 0 \\
x_{21} & x_{22} & -x_{23} \\
x_{31} & x_{23} & x_{22} 
\end{array} 
\right) \right\} . 
\end{align} 
The expression of $\PM$ is given in Proposition~\ref{PM3} as follows: 
\begin{align} 
\displaystyle \PM = \left\{ [ g_{\lambda}.\naiseki _0 ] \mid 
g_{\lambda} = \left( 
\begin{array}{ccc}
1 & 0 & 0 \\
0 & 1 & 0 \\
0 & 0 & 1/\lambda
\end{array} 
\right) , \ \lambda \geq 1 \right\} . 
\end{align} 

\begin{prop}\label{A^03} 
Let $\g=\mathfrak{r}^{\prime}_{3, a}$. 
Then, we have
\begin{enumerate}
\item 
The action of $\A$ is of cohomogeneity one, 
and $\A.\naiseki_0$ is the unique singular orbit. 
\item 
$\A.\naiseki_0$ is the unique minimal orbit. 
\end{enumerate}
\end{prop}

\begin{proof}
Take any $\naiseki$. 
In order to prove (1), 
we calculate the dimensions of the orbits. 
From the expression of $\PM$, there exists $\lambda  \geq 1$ such that 
\begin{align}
\A.\naiseki = \A.(g_{\lambda}.\naiseki_0) . 
\end{align} 
Let us denote by 
\begin{align}
U^{\prime} := g_{\lambda}^{-1} (\A) g_{\lambda} . 
\end{align} 
Then, since $g_{\lambda}^{-1}$ gives an isometry, one has an isometric congruence 
\begin{align}\label{Orb3} 
\A.(g_{\lambda}. \naiseki_0) \cong U^{\prime}.\naiseki_0 . 
\end{align}
Let $\mathfrak{u}^{\prime}$ be the Lie algebra of $U^{\prime}$. 
From the expression of $\D$, 
a direct calculation yields that 
\begin{align}\label{LAO3}
\mathfrak{u}^{\prime} 
& = g_{\lambda}^{-1} (\D) g_{\lambda} 
= \left\{ \left(
\begin{array}{ccc}
x_{11} & 0 & 0 \\
x_{21} & x_{22} & x_{23} \\
x_{31} & -\lambda^2 x_{23} & x_{33} 
\end{array}
\right) \right\} .
\end{align}
Then we have
\begin{align}
\dim \mathfrak{u}^{\prime} = 5,  \quad 
\dim (\mathfrak{u}^{\prime} \cap \mathfrak{o}(3)) = 
\begin{cases}
0 & \text{(for $\lambda > 1$)} , \\
1 & \text{(for $\lambda = 1$)} . 
\end{cases}
\end{align}
This yields that the orbit corresponding to $\lambda = 1$ 
is the unique singular orbit, 
(which has codimension two). 
This completes the proof of (1). 

We show (2). 
It is known that every singular orbit of a cohomogeneity one action is minimal 
(see \cite{Podesta}). 
Then we have only to show that $U^{\prime}.\naiseki$ is not minimal 
if $\lambda > 1$. 
From now on assume that $\lambda > 1$. 
From (\ref{LAO3}), one can see that 
\begin{align}\label{m3}
\m^{\prime} & := \mathrm{d}\pi_e (\mathfrak{u}^{\prime})
= \left\{ \left(
\begin{array}{ccc}
x_{11} & x_{21} & x_{31} \\
x_{21} & x_{22} & ((1-\lambda^2)/2) x_{23} \\
x_{31} & ((1-\lambda^2)/2) x_{23} & x_{33} 
\end{array}
\right) \right\} . 
\end{align} 
We define a basis $\{ X_1, \ldots, X_5 \}$ of $\mathfrak{u}^{\prime}$ by 
\begin{align}
\begin{split}
&
X_1 := E_{11}, \quad 
X_2 := (1/\sqrt{2}) (E_{22}+E_{33}) , \quad 
X_3 := \sqrt{2} E_{21} , \\
&
X_4 := \sqrt{2} E_{31} , \quad 
X_5 := (\sqrt{2}/(1-\lambda^2))  (E_{23}- \lambda^2 E_{32}) . 
\end{split}
\end{align}
Furthermore we put 
\begin{align}
X_i^{\prime} := (X_i)_o^\ast = (1/2) (X_i + {}^t X_i) , \quad 
A := (1/\sqrt{2}) (E_{22}-E_{33}). 
\end{align} 
Then $\{ X_1^{\prime}, \ldots , X_5^{\prime} \}$ is an orthonormal basis of $\m^{\prime}$, 
and $\{ A \}$ is an orthonormal basis of $\m \ominus \m^{\prime}$. 
Recall that the mean curvature vector $H$ is given by 
\begin{align}
H = -(1/5) 
\textstyle \sum h(X_i^\prime , X_i^\prime) , \quad 
\langle h(X_i^\prime , X_i^\prime) , A \rangle 
= \langle [A , X_i]_o^\ast , (X_i)_o^\ast \rangle . 
\end{align}
The bracket products $[A , X_i]$ satisfy 
\begin{align}
\begin{split}
&
[A , X_1] = [A , X_2] = 0 , \quad 
[A , X_3] = E_{21} , \quad 
[A , X_4] = - E_{31} , \\
& 
[A , X_5] = (2/(1-\lambda^2))(E_{23}+\lambda^2 E_{32}) . 
\end{split}
\end{align}
Hence, one has 
\begin{align}
\begin{split}
\langle [A , X_3]_o^\ast , (X_3)_o^\ast \rangle 
& 
= 1 / \sqrt{2} , \\ 
\langle [A , X_4]_o^\ast , (X_4)_o^\ast \rangle 
& 
= - 1 / \sqrt{2} , \\ 
\langle [A , X_5]_o^\ast , (X_5)_o^\ast \rangle 
& 
= \sqrt{2}(1+\lambda^2 )/(1- \lambda^2 ) . 
\end{split}
\end{align}
This yields that 
\begin{align}
H = - \sqrt{2}(1+\lambda^2 ) / (5 (1- \lambda^2)) A \neq 0 . 
\end{align}
which completes the proof. 
\end{proof}

\section*{Acknowledgement}
The authors would like to thank Yoshio Agaoka and Kazuhiro Shibuya for useful comments and suggestions. 
The first author was supported in part by Grant-in-Aid for Japan Society for the Promotion of Science Fellows (11J05284). 
The second author was supported in part by 
JSPS KAKENHI Grant Numbers 24654012, 26287012.


\begin{thebibliography}{99}

\bibitem{And}
Arvanitoyeorgos, A.: 
An introduction to Lie groups and the geometry of homogeneous spaces. 
Student Mathematical Library, {\bf 22}. American Mathematical Society Providence, RI, 2003. 

\bibitem{ABDO}
Andrada, A., Barberis, L., Dotti, I., Ovando,G.: 
Product structures on four dimensional solvable Lie algebras. 
Homology Homotopy Appl.\ \textbf{7} (2005), no. 1, 9--37.


\bibitem{Bes}
Besse, A.: 
Einstein manifolds. Classics in Mathematics, Springer-Verlag, Berlin, 2008. 

\bibitem{F}
Fern{\accent 19 a}ndez-Culma, E. A. 
Classification of 7-dimensional Einstein nilradicals. 
Transform.\ Groups \textbf{17} (2012), no. 3, 639--656. 


\bibitem{HL} 
Ha, K.~Y., Lee, J.~B.: 
Left invariant metrics and curvatures on simply connected three-dimensional 
Lie groups. 
Math.\ Nachr. \textbf{282} (2009), 868--898. 

\bibitem{HTT}
Hashinaga, T., Tamaru, H., Terada, K.: 
Milnor-type theorems for left-invariant Riemannian metrics on Lie groups. 
J.\ Math.\ Soc.\ Japan, to appear. Preprint at arXiv:1501.02485


\bibitem{Heber}
Heber, J.: 
Noncompact homogeneous Einstein spaces. 
Invent.\ Math.\ \textbf{133} (1998), no. 2, 279--352. 

\bibitem{Hel}
Helgason, S.: 
Differential geometry, Lie groups, and symmetric spaces. 
Graduate Studies in Mathematics \textbf{34}, 
American Mathematical Society, Providence, RI, 2001. 

\bibitem{J}
Jablonski, M.: 
Concerning the existence of Einstein and Ricci soliton metrics on solvable Lie groups.  Geom.\ Topol.\ \textbf{15} (2011), no. 2, 735--764. 

\bibitem{J11}
Jablonski, M.: 
Homogeneous Ricci solitons. 
J.\ Reine Angew.\ Math., to appear. 

\bibitem{KTT}
Kodama, H., Takahara, A., Tamaru, H.: 
The space of left-invariant metrics on a Lie group up to isometry and scaling. 
Manuscripta Math.\ \textbf{135} (2011), no. 1--2, 229--243. 

\bibitem{K-T}Kubo, A., Tamaru, H.: 
A sufficient condition for congruency of orbits of Lie groups and some applications. 
Geom.\ Dedicata
\textbf{167} (2013), no.~1, 233--238. 

\bibitem{Lau01}
Lauret, J.: 
Ricci soliton homogeneous nilmanifolds. 
Math.\ Ann.\ \textbf{319} (2001), no. 4, 715--733. 


\bibitem{Lau03}
Lauret, J.: 
Degenerations of Lie algebras and geometry of Lie groups. 
Differential Geom.\ Appl.\ \textbf{18} (2003), no. 2, 177--194.

\bibitem{Lau09}
Lauret, J.: 
Einstein solvmanifolds and nilsolitons. 
New developments in Lie theory and geometry, 1--35, Contemp.\ Math.\ \textbf{491}, 
Amer. Math. Soc., Providence, RI, 2009. 

\bibitem{Lau10}
Lauret, J.: 
Einstein solvmanifolds are standard. 
Ann.\ of Math.\ (2) \textbf{172} (2010), no. 3, 1859--1877. 

\bibitem{Lau11}
Lauret, J.: 
Ricci soliton solvmanifolds. 
J.\ Reine Angew.\ Math.\ \textbf{650} (2011), 1--21. 

\bibitem{L-W}
Lauret, L., Will, C.: 
Einstein solvmanifolds: existence and non-existence questions. 
Math.\ Ann.\ \textbf{350} (2011), no. 1, 199--225. 

\bibitem{Mil}
Milnor, J.: 
Curvatures of left invariant metrics on Lie groups. 
Advances in Math.\ \textbf{21} (1976), no. 3, 293--329. 

\bibitem{N06}
Nikolayevsky, Y.: 
Einstein solvmanifolds with free nilradical. 
Ann.\ Global Anal.\ Geom.\ \textbf{33} (2008), no. 1, 71--87. 

\bibitem{N07}
Nikolayevsky, Y.: 
Einstein solvmanifolds with a simple Einstein derivation. 
Geom.\ Dedicata \textbf{135} (2008), 87--102. 

\bibitem{Podesta}
Podest{\accent 18 a}, F.: 
Some remarks on austere submanifolds. 
Boll.\ Un.\ Mat.\ Ital.\ B(7) \textbf{11} (1997), no.\ 2, suppl., 157--160. 

\bibitem{T05}
Tamaru, H.: 
A class of noncompact homogeneous Einstein manifolds. 
Differential geometry and its applications, 119--127, Matfyzpress, Prague, 2005. 

\bibitem{T08}
Tamaru, H.: 
Noncompact homogeneous Einstein manifolds attached to graded Lie algebras. 
Math.\ Z.\ \textbf{259} (2008), no. 1, 171--186. 

\bibitem{T11}
Tamaru, H.: 
Parabolic subgroups of semisimple Lie groups and Einstein solvmanifolds. 
Math. Ann.\ \textbf{351} (2011), no. 1, 51--66. 

\bibitem{W03}
Will, C.: 
Rank-one Einstein solvmanifolds of dimension 7. 
Differential Geom.\ Appl.\ \textbf{19} (2003), no. 3, 307--318. 

\bibitem{W11}
Will, C.: 
The space of solvsolitons in low dimensions. 
Ann.\ Global Anal.\ Geom.\ \textbf{40} (2011), no. 3, 291--309.
\end{thebibliography}
\end{document}